\documentclass[12pt]{amsart}
\usepackage{extpfeil}
\usepackage{extarrows}

\usepackage{mathrsfs, amssymb, array, amsmath, upgreek, amsfonts, mathtools, mathcomp, longtable}
\usepackage{amsthm}
\usepackage{xcolor,graphicx}

\usepackage[matrix,arrow,curve]{xy}

\newcommand{\fP}{\mathfrak{P}}

\usepackage{hyperref}

\DeclareMathOperator{\Sing}{Sing}

\newcommand{\type}[1]{$\mathsf{X}_{(#1)}$}
\newcommand{\mtype}[1]{X_{(#1)}}

\newcommand{\ppl}{+\!+}

\newcommand{\fS}{\mathfrak{S}}
\newcommand{\fA}{\mathfrak{A}}

\newcommand{\rC}{\mathrm{C}}
\newcommand{\rD}{\mathrm{D}}
\newcommand{\rF}{\mathrm{F}}
\newcommand{\rR}{\mathrm{R}}
\newcommand{\rV}{\mathrm{V}}

\newcommand{\rh}{\mathrm{h}}

\newcommand{\Db}{\mathbf{D}}

\newcommand{\hY}{{\hat{Y}}}

\newcommand{\barV}{{\bar{V}}}

\newcommand{\ZZ}{\mathbb{Z}}
\newcommand{\PP}{\mathbb{P}}
\newcommand{\QQ}{\mathbb{Q}}

\newcommand{\cO}{{\mathscr{O}}}
\newcommand{\cX}{{\mathscr{X}}}
\newcommand{\MMM}{{\mathscr{M}}}
\newcommand{\xref}[1]{\textup{\ref{#1}}}
\newcommand{\g}{{\mathrm{g}}}
\newcommand{\Pic }{\operatorname{Pic}}

\newcommand{\kk}{{\mathsf{k}}}
\newcommand{\bkk}{{\bar{\mathsf{k}}}}
\newcommand{\Gal}{{\mathrm{G}}}

\newcommand{\rG}{{\mathrm{G}}}

\newcommand{\rT}{{\mathrm{T}}}
\newcommand{\rX}{{\mathrm{X}}}
\newcommand{\rY}{{\mathrm{Y}}}
\newcommand{\rx}{{\mathrm{x}}}
\newcommand{\ry}{{\mathrm{y}}}

\DeclareMathOperator{\Gr}{\mathrm{Gr}}
\DeclareMathOperator{\Aut}{\mathrm{Aut}}
\DeclareMathOperator{\PGL}{\mathrm{PGL}}
\newcommand{\Gm}{\mathbb{G}_{\mathrm{m}}}
\newcommand{\tC}{{\tilde{C}}}
\newcommand{\tX}{{\tilde{X}}}
\newcommand{\tY}{{\tilde{Y}}}
\newcommand{\tD}{{\tilde{D}}}

\DeclareMathOperator{\Sch}{\mathrm{Sch}}
\DeclareMathOperator{\Hilb}{\mathrm{Hilb}}
\DeclareMathOperator{\Bl}{\mathrm{Bl}}
\DeclareMathOperator{\Res}{\mathrm{Res}}

\newcommand{\cK}{{\mathscr{K}}}

\DeclareMathOperator{\Sym}{\mathrm{Sym}}
\DeclareMathOperator{\Ker}{\mathrm{Ker}}

\DeclareMathOperator{\Ima}{\mathrm{Im}}
\DeclareMathOperator{\Cl}{\mathrm{Cl}}
\DeclareMathOperator{\Spec}{\mathrm{Spec}}

\newcommand{\cN}{{\mathscr{N}}}

\newcommand{\cV}{{\mathscr{V}}}

\newcommand{\bL}{\mathbf{L}}
\newcommand{\bR}{\mathbf{R}}

\newcommand{\rc}{{\mathrm{c}}}
\newcommand{\cE}{{\mathscr{E}}}
\newcommand{\cF}{{\mathscr{F}}}

\DeclareMathOperator{\rk}{\mathrm{rk}}

\newcommand{\whi}{{\widehat{\imath}}}

\makeatletter
\@addtoreset{equation}{subsection}
\makeatother

\theoremstyle{plain}

\newtheorem{theorem}{Theorem}[section]
\newtheorem{lemma}[theorem]{Lemma}

\newtheorem{proposition}[theorem]{Proposition}
\newtheorem{corollary}[theorem]{Corollary}

\newtheorem{conjecture}[theorem]{Conjecture}

\theoremstyle{definition}

\newtheorem*{definition*}{Definition}
\newtheorem{example-remark}{Remark-Example}

\newtheorem*{notation*}{Notation}

\newtheorem{remark}[theorem]{Remark}

\newcounter{NN}

\title{Rationality over non-closed fields of Fano threefolds\\[1ex] with higher geometric Picard rank}
\author{Alexander Kuznetsov}
\thanks{The paper was partially supported by the HSE University Basic Research Program.
}
\address{\parbox{0.9\textwidth}{Steklov Mathematical Institute of Russian Academy of Sciences, Moscow, Russia 
\\[1pt]
Laboratory of Algebraic Geometry, NRU HSE, Moscow, Russia\\}}
\email{akuznet@mi-ras.ru} 

\author{Yuri Prokhorov}

\address{\parbox{0.9\textwidth}{Steklov Mathematical Institute of Russian Academy of Sciences, Moscow, Russia 
\\[1pt]
Laboratory of Algebraic Geometry, NRU HSE, Moscow, Russia
\\[1pt]
Department of Algebra, Moscow State University, Moscow, Russia
\\}}
\email{prokhoro@mi-ras.ru}

\date{}

\begin{document} 

\begin{abstract}
We prove rationality criteria over non-closed fields of characteristic~$0$
for five out of six types of geometrically rational Fano threefolds of Picard number~$1$ and geometric Picard number bigger than~$1$.
For the last type of such threefolds we provide a unirationality criterion 
and construct examples of unirational but not stably rational varieties of this type.
\end{abstract}
\maketitle

\section{Introduction}

\subsection{The results}
\label{ov}

The goal of this paper is to discuss rationality of smooth Fano threefolds over algebraically non-closed fields of characteristic~$0$.
In~\cite{KP19} we considered the case of geometrically rational Fano threefolds with geometric Picard number~$\uprho(X_\bkk) = 1$ 
and here we switch the focus to the case of geometrically rational Fano threefolds~$X$ with Picard numbers 
\begin{equation}
\label{eq:rho-assumptions}
\uprho(X) = 1
\qquad\text{and}\qquad
\uprho(X_\bkk) > 1.
\end{equation}
In fact, Fano threefolds satisfying~\eqref{eq:rho-assumptions} have been classified in~\cite{Prokhorov-GFano-2}, 
and~\cite{Alzati-Bertolini-1992a} explains which of these are geometrically rational.
A combination of these results gives the following 

\begin{theorem}[{\cite[Theorem~1.2]{Prokhorov-GFano-2}}, {{\cite{Alzati-Bertolini-1992a}}}]
\label{thm:classif}
There are exactly six families of geometrically rational Fano threefolds satisfying~\eqref{eq:rho-assumptions}
as listed in Table~\xref{table:fanos}.
\begin{table}[h]
\begin{tabular}
{lcccccp{0.5\textwidth}}
\label{ta:res1}
\\\hline
\\[-2ex]
&$\upiota(X_{\bkk})$&$\uprho(X_{\bkk})$
&$-K_{X}^3$
&$\g(X)$
&$\rh^{1,2}(X_{\bkk})$\raisebox{-1.5ex}{\ }
&\multicolumn{1}{c}{$X_{\bkk}$}\\
\hline
\\[-2ex]
\type{3,3}
& $1$&$2$& $20$ & $11$ & $3$ &
an intersection of three divisors of bidegree $(1,1)$ in~$\PP_{\bkk}^3\times \PP_{\bkk}^3$;
\\ [5pt]
\type{1,1,1,1}
& $1$
& $4$& $24$& $13$ & $1$ &
a divisor of multidegree $(1,1,1,1)$ in~$(\PP_{\bkk}^1)^4$;
\\ [5pt]
\type{4,4}
& $1$& $2$& $28$& $15$ & $0$ &
the blow-up of a
smooth quadric $Q_{\bkk} \subset \PP_{\bkk}^4$ along a linearly normal smooth rational quartic curve;
\\ [5pt]
\type{2,2,2}
& $1$
& $3$& $30$& $16$ & $0$ &	
an intersection of three divisors 
of multidegrees~$(0,1,1)$, $(1,0,1)$, $(1,1,0)$ in~$\PP_{\bkk}^2\times \PP_{\bkk}^2\times \PP_{\bkk}^2$;
\\
\type{2,2}
& $2$
&$2$& $48$& $25$ & $0$ &
a divisor of bidegree $(1,1)$ in~$\PP_{\bkk}^2\times \PP_{\bkk}^2$;
\\ [5pt]
\type{1,1,1}
& $2$
& $3$& $48$ & $25$ & $0$ &
$\PP_{\bkk}^1\times \PP_{\bkk}^1\times \PP_{\bkk}^1$.\raisebox{-1.5ex}{\ } 
\\\hline
\end{tabular}
\smallskip
\caption{Geometrically rational Fano threefolds~$X$ 
satisfying~\eqref{eq:rho-assumptions}}
\label{table:fanos}
\end{table}
\end{theorem} 

The first column of Table~\ref{table:fanos} contains the name for the family we use in this paper, 
the next columns contain the index~$\upiota(X_\bkk)$, defined as
\begin{equation*}
\upiota(X_\bkk) = \max\left\{ i \ \left|\ \tfrac1i K_{X_\bkk} \in \Pic(X_\bkk) \right.\right\},
\end{equation*}
the geometric Picard number~$\uprho(X_\bkk)$, the anticanonical 
degree~$(-K_X)^3$, the genus~$\g(X)$, defined by
\begin{equation*}
(-K_X)^3= 2\g(X) - 2
\end{equation*}
and the Hodge number $\rh^{1,2}(X_\bkk)$ of the threefold,
while the last column provides a geometric description of these varieties
over an algebraic closure~$\bkk$ of the base field.

We discuss some geometric properties of threefolds from Table~\ref{table:fanos} in~\S\ref{sec:contractions}.
In particular, we describe their extremal contractions over~$\bkk$ and identify their Hilbert schemes of lines and conics,
as well as the subschemes of the Hilbert schemes of twisted cubic curves passing through a general point.

However, our main interest is in rationality criteria, and the next theorem is our main result.

\begin{theorem}
\label{thm:unirat}
Let $X$ be a Fano threefold from Table~\xref{table:fanos};
in particular we assume~$\uprho(X) = 1$.
\begin{enumerate}
\item 
\label{thm:unirat:unirat}
$X$ is unirational if and only if~$X(\kk) \ne \varnothing$.
\item 
\label{thm:unirat:rat}
If $X$ has type~\type{4,4}, \type{2,2,2}, \type{2,2}, or~\type{1,1,1} 
then~$X$ is $\kk$-rational if and only if $X(\kk) \ne \varnothing$.
\item 
\label{thm:unirat:nonrat}
If $X$ has type~\type{3,3} then~$X$ is never 
$\kk$-rational.
\end{enumerate}
\end{theorem}

Note that over an algebraically closed field threefolds of types~\type{4,4}, \type{2,2,2}, \type{2,2}, 
and~\type{1,1,1} have~$\rh^{1,2} = 0$, hence trivial intermediate Jacobians, while the intermediate Jacobians of threefolds 
of types~\type{3,3} and~\type{1,1,1,1} over~$\bkk$ are Jacobians of curves of genus~$3$ and~$1$, respectively
(and $\kk$-forms of these over~$\kk$);
this explains the difference in the behavior.

It is a classical fact that the existence of a $\kk$-point is necessary for rationality or unirationality,
so the major part of the proof of the theorem consists of proving rationality or unirationality under this assumption.
We use for this a case-by-case analysis 
(see~\S\ref{subsec:sketch} for a description of our approach).
The theorem is thus a combination of the following results (we assume everywhere~$X(\kk) \ne \varnothing$):

\begin{itemize}
\item 
rationality for threefolds of type~\type{1,1,1} is proved in Corollary~\ref{cor:x111};
\item 
rationality for threefolds of type~\type{2,2} is proved in Proposition~\ref{prop:x22};
\item 
rationality for threefolds of type~\type{2,2,2} is proved in Proposition~\ref{prop:x222};
\item 
rationality for threefolds of type~\type{4,4} is proved in Proposition~\ref{prop:x44};
\item 
unirationality for threefolds of type~\type{1,1,1,1} is proved in Proposition~\ref{prop:x1111};
\item 
unirationality for threefolds of type~\type{3,3} is proved in Proposition~\ref{prop:x33};
\item 
non-rationality for threefolds of type~\type{3,3} is proved in Corollary~\ref{cor:x33}.
\end{itemize}

Theorem~\ref{thm:unirat} provides nice criteria for rationality of the five out of six types of Fano threefolds listed in Table~\ref{table:fanos}.
For the remaining type~\type{1,1,1,1} we have a conjecture and a partial result.

\begin{conjecture}
\label{conj:x1111}
If $X$ has type~\type{1,1,1,1} and~$\uprho(X) = 1$ then~$X$ is never $\kk$-rational.
\end{conjecture}

To explain the partial result we need to introduce some notation.
Let~$X$ be a Fano threefold of type~\type{1,1,1,1}.
As we show in Lemma~\ref{lemma:picard}, the action of the Galois group~$\Gal(\bkk/\kk)$ on~$\Pic(X_\bkk)$ 
factors through the group~$\fS_4$ that acts by permutations of the pullbacks of the point classes of the factors of the ambient~$(\PP_{\bkk}^1)^4$,
and the assumption~$\uprho(X) = 1$ means that the subgroup
\begin{equation*}
\rG_X := \Ima(\Gal(\bkk/\kk) \longrightarrow \fS_4) \subset \fS_4
\end{equation*}
is \emph{transitive},
hence belongs to the following list of (conjugacy classes) of transitive subgroups of~$\fS_4$:
\begin{equation*}
\label{eq:transitive-subgroups}
\rG_X \in \{ \fS_4, \fA_4, \rD_4, \rV_4, \rC_4 \},
\end{equation*} 
where~$\fA_4$ is the alternating subgroup, 
$\rD_4$ is the dihedral group of order~8 (a Sylow 2-subgroup in~$\fS_4$),
$\rV_4$ is the Klein group of order~$4$,
and~$\rC_4$ is the cyclic group of order~$4$.
Note that all of these groups contain~$\rV_4$ except for~$\rC_4$.

\begin{theorem}
\label{thm:x1111-non-st-rat}
Let~$\rG \subset \fS_4$ be a subgroup containing the Klein group~$\rV_4 \subset \fS_4$.
Let~$\kk$ be a field such that there is an epimorphism~$\Gal(\bkk/\kk) \twoheadrightarrow \rG$.
Then for the field of rational functions~$K = \kk(t)$ there exists a variety~$X$ over~$K$ of type~\type{1,1,1,1} 
such that~$\rG_X = \rG$, $\uprho(X) = 1$, and~$X(K) \ne \varnothing$, 
but~$X$ is not stably rational over~$K$.
\end{theorem}

\subsection{The proofs}
\label{subsec:sketch} 

For (uni)rationality constructions it is natural to use $\kk$-Sarkisov links:
\begin{equation}
\label{eq:sl-general}
\vcenter{
\xymatrix{
& 
\tilde X \ar[dl]_{\sigma}\ar@{-->}^{\psi}[rr] \ar[dr]^{\phi} 
&&
\tilde X^+ \ar[dr]^{\sigma_+} \ar[dl] \ar[dl]_{\phi_+}
\\
X &&
\bar X && 
X^+,
} }
\end{equation} 
where $\sigma$ is the blowup of a $\kk$-irreducible subvariety, 
$\phi$ and $\phi_+$ are small crepant birational contractions, 
$\psi$ is a flop, and 
$\sigma_+$ is a Mori extremal contraction.
Note that such a link is completely determined by the center of the blowup~$\sigma$ --- 
the contractions and the flop are obtained by the $\kk$-Minimal Model Program applied to~$\tX$
(note that~$\uprho(\tX) = 2$, so the output of the MMP is unambiguous);
in particular the link is defined over~$\kk$.
For our purpose it is enough to consider two types of Sarkisov links:
\begin{itemize}
\item 
Sarkisov links where~$\sigma$ is the blowup of a $\kk$-point;
\item 
Sarkisov links where~$\sigma$ is the blowup of a reduced $\kk$-irreducible singular conic.
\end{itemize}
We construct the corresponding links accurately for threefolds of type~\type{4,4} in~\S\ref{sec:x44} (see Theorem~\ref{theorem:x44-links})
by using standard MMP arguments.
Of course, a similar construction could be given for other types of Fano threefolds from Table~\ref{table:fanos},
but to make the argument less tedious we use the fact that all other among these threefolds are $\kk$-forms 
of complete intersections in products of projective spaces and deduce the required (uni)rationality constructions 
from an appropriate birational transformation for a product of projective spaces.

With this goal in mind we construct in~\S\ref{sec:toric}
a toric birational transformation between the product~$(\PP^n)^r$ of projective spaces 
and a $\PP^r$-bundle over the product~$(\PP^{n-1})^r$ of smaller projective spaces, 
see Theorem~\ref{proposition:toric-link}
(in fact, we construct a birational transformation in a slightly more general situation, 
but the setup described above is the only one that we need for applications in the paper).
This theorem has a consequence of independent interest, Corollary~\ref{corollary:product-rational}, 
saying that a $\kk$-form of a product of projective spaces is $\kk$-rational if and only if it has a $\kk$-point.
This corollary immediately gives the required rationality construction for Fano threefolds of type~\type{1,1,1} (Corollary~\ref{cor:x111}),
and with a bit of more work provides rationality constructions for threefolds of types~\type{2,2} (Proposition~\ref{prop:x22}) 
and~\type{2,2,2} (Proposition~\ref{prop:x222})
as well as a unirationality construction for threefolds of type~\type{1,1,1,1} (Proposition~\ref{prop:x1111}). 

In the case of a variety~$X$ of type~\type{3,3} with a $\kk$-point~$x$ 
we again use the toric transformation of Theorem~\ref{proposition:toric-link}
to construct a birational equivalence of~$X$ with a divisor~$X^+$ of bidegree~$(2,2)$ in a $\kk$-form of~$\PP^2 \times \PP^2$.
If~$x$ lies on a $\bkk$-line in~$X$, we check that~$X^+$ contains a $\kk$-form of the quadric surface~$\PP^1 \times \PP^1$
and use this to deduce unirationality of~$X$ (Proposition~\ref{prop:x33}).
If~$x$ does not lie on a line, we check in Proposition~\ref{prop:conic-bundle-x33}
that~$X^+$ described above is, in fact, the mid-point of a Sarkisov link,
that ends with a conic bundle over~$\PP^2$ which has a smooth quartic curve~$\Gamma \subset \PP^2$ as discriminant.
We also check that the discriminant double covering~$\tilde\Gamma \to \Gamma$ 
associated to
this conic bundle
is trivial over a quadratic extension~$\kk'$ of the base field~$\kk$ but nontrivial over~$\kk$,
and that the conic bundle has a rational section over~$\kk'$.
We check in Theorem~\ref{prop:non-rationality-conic-bundle} that these geometric properties 
characterize the non-rational conic bundles constructed by Benoist and Wittenberg in~\cite{BW}
and deduce in Corollary~\ref{cor:x33} non-rationality of~$X$ from~\cite[Proposition~3.4]{BW}. 

In the last part of the paper, \S\ref{sec:x1111}, we discuss Fano threefolds of type~\type{1,1,1,1}.
To prove Theorem~\ref{thm:x1111-non-st-rat} we use a degeneration technique.
Namely we construct a family of Fano threefolds of type~\type{1,1,1,1} over~$\PP^1_\kk$
with special fiber a singular toric threefold (with ordinary double points) 
which is well-known not to be stably rational.
Since stable rationality is specialization-closed by a result of Nicaise and Shinder~\cite{NS}, 
we conclude that the general fiber of the constructed family is also not stably rational.

\subsection*{Acknowledgements.} 

We would like to thank Sergey Gorchinskiy, Zhenya Shinder
and Costya Shramov for useful discussions. We are also grateful to the anonymous
referee for correcting a mistake in the original statement of Theorem~\ref{thm:x1111-non-st-rat} 1.4 and for useful
comments. This work was performed at the Steklov International Mathematical Center
and supported by the Ministry of Science and Higher Education of the Russian Federation
(agreement no. 075-15-2022-265). The paper was also partially supported by the HSE
University Basic Research Program.


\section{Extremal contractions and Hilbert schemes of curves}
\label{sec:contractions} 

In this section we describe the geometry of Fano threefolds of index~$1$ from Table~\ref{table:fanos}.
In particular, we describe their extremal contractions over~$\bkk$ 
as well as their Hilbert schemes of lines and conics, and Hilbert schemes of twisted cubic curves passing through a fixed point.

To start with, recall that for most Fano threefolds the anticanonical linear system is very ample and 
the anticanonical image is an intersection of quadrics;
in fact Fano threefolds which do not enjoy these nice properties (hyperelliptic and trigonal ones)
have been classified and listed in~\cite{Iskovskikh-1980-Anticanonical}.
It is easy to check that Fano threefolds from Table~\ref{table:fanos} are not in this list;
therefore we obtain 

\begin{theorem}
[{\cite[Chapter~2, Theorems~2.2 and~3.4]{Iskovskikh-1980-Anticanonical}}]
\label{th:bht}
Let $X$ be a Fano threefold from Table~\textup{\ref{table:fanos}}.
The anticanonical class $-K_X$ is very ample and the anticanonical image 
\begin{equation*}
X = X_{2g-2} \subset \PP^{g+1}
\end{equation*}
is an intersection of quadrics \textup(as a scheme\textup), where $g = 
\g(X)$.
\end{theorem}

\subsection{Contractions over~$\bkk$}
\label{subsec:contractions}

Assume~$X$ is a Fano threefold of index~$1$ from Table~\ref{table:fanos}, i.e., 
a threefold of either of types~\type{2,2,2}, \type{4,4}, \type{3,3}, 
\type{1,1,1,1}.
Then there is an embedding 
\begin{equation}
\label{eq:x-y}
X_\bkk \subset Y \cong (\PP^n)^r, 
\end{equation}
(we will see in Lemma~\ref{lemma:picard} that~$r = \uprho(X_\bkk)$, hence the notation), where
\begin{equation*}
(n,r) = (2,3),\ (4,2),\ (3,2),\ \text{or}\ (1,4).
\end{equation*}
Indeed, for types~\type{2,2,2}, \type{3,3}, \type{1,1,1,1} this holds by definition 
and for type~\type{4,4} this follows from the following 

\begin{lemma}
\label{lemma:X44}
Let $\Gamma_1 \subset Q_1 \subset \PP^4$ be a linearly normal smooth rational quartic curve in a smooth quadric threefold.
If~$H_1$ is the hyperplane class of~$Q_1$ then the linear system~$|2H_1 - \Gamma_1|$ of quadrics through~$\Gamma_1$ 
defines a birational morphism $\pi_2 \colon \Bl_{\Gamma_1}Q_1 \to Q_2 \subset \PP^4$ onto 
another smooth quadric threefold~$Q_2$ 
and this morphism is itself the blowup of a linearly normal smooth rational 
quartic curve~$\Gamma_2 \subset Q_2$, 
so that
\begin{equation*}
\Bl_{\Gamma_1}(Q_1) \cong \Bl_{\Gamma_2}(Q_2).
\end{equation*}
Moreover, if~$X$ is a Fano threefold of type~\type{4,4} there is a natural embedding 
\begin{equation*}
X_\bkk \xhookrightarrow{\hspace{1.1em}} Q_1 \times Q_2 \subset \PP^4 \times \PP^4
\end{equation*}
such that~$-K_{X_\bkk}$ is the sum of the pullbacks of the hyperplane classes of 
the factors.
\end{lemma}
\begin{proof}
The curve~$\Gamma_1$ is an intersection of six quadrics in~$\PP^4$; therefore it is an intersection of five quadrics in~$Q_1$. 
Hence, if~$E_1$ is the exceptional divisor of the blowup~$\pi_1 \colon X_{\bkk} \to Q_1$ 
and~$H_1$ is the pullback of the hyperplane class of~$Q_1$, 
the linear system~$|2H_1-E_1|$ on~$X_\bkk$ is 4-dimensional and base point free.
Therefore, this linear system defines a morphism~$\pi_2 \colon \Bl_{\Gamma_1}(Q_1)\to \PP^4$;
moreover, standard intersection theory gives~$(2H_1 - E_1)^3 = 2$.
Hence the image of~$\pi_2$ (which is not contained in a hyperplane by definition)
is a quadric~\mbox{$Q_2 \subset \PP^4$} and~$\pi_2$ is birational. 
Since $-K_{X_\bkk}$ is ample on the fibers of~$\pi_2$ and $\uprho(X_\bkk) = 2$, 
we see that~$\pi_2$ is an extremal Mori contraction.
By \cite{Mori-1982} the quadric~$Q_2$ is smooth and~$\pi_2$ is the blowup of a curve which must be a linearly normal smooth rational quartic curve.
For the last statement just note that~$H_1 + (2H_1 - E_1) = 3H_1 - E_1$ is the anticanonical class of~$X_\bkk$.
\end{proof}

We denote by~$H_i$, $1 \le i \le r$, the pullbacks to~$Y = (\PP^n)^r$ of the hyperplane classes of the factors and,
abusing the notation, also their restrictions to~$X_\bkk$ via the embedding~\eqref{eq:x-y}.

\begin{lemma}
\label{lemma:picard-simple}
If~$X$ is a threefold of either of 
types~\type{2,2,2}, \type{4,4}, \type{3,3}, \type{1,1,1,1} then
the Picard group~$\Pic(X_\bkk)$ is freely generated by the classes~$H_i$:
\begin{equation*}
\Pic(X_\bkk) = \bigoplus_{i=1}^r \ZZ H_i.
\end{equation*}
Moreover,
\begin{equation}
\label{eq:kx}
-K_{X_\bkk} = H := H_1 + \dots + H_r;
\end{equation}
\end{lemma}
\begin{proof}
For type~\type{4,4} this follows from Lemma~\ref{lemma:X44}, 
and for the other types the first statement follows from the Lefschetz Hyperplane Theorem 
and the second from adjunction and the description of Table~\ref{table:fanos}.
\end{proof}

For each subset $I \subset \{1,\dots,r\}$ we consider the projection
\begin{equation}
\label{eq:pi-i}
\pi_I \colon X_\bkk \xhookrightarrow{\hspace{1.1em}} 
Y \longrightarrow \prod_{i \in I} \PP^n \cong (\PP^n)^{|I|}.
\end{equation}
Especially useful are the morphisms~$\pi_I$ for $I$ of cardinality~$r - 1$, 
so we introduce the notation
\begin{equation*}
\whi := \{1,\dots,r\} \setminus \{i\}
\end{equation*}
and write 
\begin{equation}
\label{eq:pi-hi}
\pi_\whi \colon X_\bkk \longrightarrow (\PP^n)^{r-1}
\end{equation} 
for the corresponding morphisms.
Note that in the case~$r = 2$ we have~$\whi = \{ 3 - i \}$, so these morphisms are the same as morphisms~$\pi_{3-i}$.
The next lemma describes~$X_\bkk$ in terms of the~$\pi_\whi$.

\begin{lemma}
\label{lemma:pi-i}
The morphism~$\pi_\whi$ is birational onto its image and the exceptional divisor~$E_\whi$ of~$\pi_\whi$ is irreducible.
More precisely, the morphism~$\pi_\whi$ identifies~$X_\bkk$ as follows:
\begin{enumerate}
\item\label{lemma:pi-i:222}
if $X$ has type~\type{2,2,2} the map~$\pi_\whi$ is the blowup of a smooth divisor~$W_\whi \subset \PP^2 \times \PP^2$ of bidegree~$(1,1)$ 
along a smooth rational curve~$\Gamma_\whi \subset W_\whi$ of bidegree~$(2,2)$ whose projections to the factors~$\PP^2$ are closed embeddings;
the divisor class~$H_i$ is equal to~$\sum_{j\ne i}H_j - E_\whi$;
\item\label{lemma:pi-i:44}
if $X$ has type~\type{4,4} the map~$\pi_i$ is the blowup of a $3$-dimensional 
quadric~$Q_{i}$ along a smooth linearly normal rational curve~$\Gamma_{i} \subset Q_{i}$ of degree~$4$;
the divisor class~$H_{\whi}$ is equal to~$2H_{i} - E_i$;
\item\label{lemma:pi-i:33}
if $X$ has type~\type{3,3} the map~$\pi_i$ is the blowup of~$\PP^3$ along a 
smooth curve~$\Gamma_{i} \subset \PP^3$ of genus~$3$ and degree~$6$;
the divisor class~$H_{\whi}$ is equal to~$3H_{i} - E_i$; 
\item
\label{lemma:pi-i:x1111}
if $X$ has type~\type{1,1,1,1} the map~$\pi_\whi$ is the blowup of~$\PP^1 \times \PP^1 \times \PP^1$ 
along a smooth elliptic curve~$\Gamma_\whi \subset (\PP^1)^3$ of multidegree~$(2,2,2)$;
the divisor class~$H_i$ is equal to~$\sum_{j\ne i}H_j - E_\whi$.
\end{enumerate}
\end{lemma}
\begin{proof}
Part~\ref{lemma:pi-i:44} is proved in Lemma~\ref{lemma:X44}.
So, assume~$X$ is a variety of either of types~\type{2,2,2}, \type{3,3}, or~\type{1,1,1,1}.
Birationality of the projection~$\pi_\whi$ is clear from the descriptions of Table~\ref{table:fanos};
it also follows that all fibers of~$\pi_\whi$ are linear subspaces in~$\PP^n$
and~$-K_{X_\bkk}$ restricts to each of them as the hyperplane class by~\eqref{eq:kx}.
Also, it is easy to see that the image of~$\pi_\whi$ is smooth in all cases
(for type~\type{2,2,2} if~$W_\whi \subset \PP^2 \times \PP^2$ is singular 
then its preimage in~$\PP^2 \times \PP^2 \times \PP^2$ is singular along a plane, 
hence~$X_\bkk$, which is the intersection of this preimage with two other divisors, must be singular;
and for types~\type{3,3} and~\type{1,1,1,1} the image is just~$\PP^3$ or~$\PP^1 \times \PP^1 \times \PP^1$, respectively).

By Lemma~\ref{lemma:picard-simple} the relative Picard number of~$\pi_\whi$ 
is~$1$ and~$-K_{X_\bkk}$ is ample, hence~$\pi_\whi$ is an extremal Mori contraction.
Since both the source and target of~$\pi_\whi$ are smooth, 
it follows from~\cite{Mori-1982} that the morphism~$\pi_\whi$ is either the blowup of a smooth curve or the blowup of a smooth point.
In the latter case the restriction of~$-K_{X_\bkk}$ to the nontrivial fiber~$\PP^2$ of~$\pi_\whi$ would be isomorphic to~$\cO_{\PP^2}(2)$,
contradicting to the above observation, hence~$\pi_\whi$ is the blowup of a smooth curve.

The remaining assertions 
are easy and left to the reader (see also \cite{Mori1981-82}).
\end{proof}

\begin{lemma}
\label{lemma:picard}
The classes~$H_i$ are semiample and generate the nef cone of~$X_\bkk$.
The Galois group~$\Gal(\bkk/\kk)$ permutes these classes in a transitive way.
In other words, the natural homomorphism~$\varpi_X \colon \Gal(\bkk/\kk) \to \Aut(\Pic(X_\bkk))$ 
factors through the subgroup~$\fS_r \subset \Aut(\Pic(X_\bkk))$
and its image
\begin{equation}
\label{eq:gx}
\rG_X := \Ima(\Gal(\bkk/\kk) \xrightarrow{\ \varpi_X\ } \fS_r)
\end{equation}
is a transitive subgroup of~$\fS_r$.
\end{lemma}

\begin{proof}
The classes~$H_i$ are pullbacks of ample classes on~$\PP^n$, hence semiample,
and they generate~$\Pic(X_\bkk)$ by Lemma~\ref{lemma:picard-simple}.
If~$\Lambda_i$ is the class of a non-trivial fiber of~$\pi_\whi$, we have 
\begin{equation*}
H_j \cdot \Lambda_i= \delta_{ij},
\end{equation*}
therefore~$H_j$ generate the rays of the nef cone.
It follows that the Galois group permutes the~$H_i$, hence its action on~$\Pic(X_\bkk)$ factors through the permutation group.
Transitivity of the subgroup~$\rG_X \subset \fS_r$ follows from the equality~$\uprho(X) = 1$.
\end{proof} 

We say that a surface~$\Pi \subset X_\bkk$ is an {\sf $H$-plane} if~$\Pi \cong \PP^2_\bkk$ and~$H\vert_\Pi$ is the line class.

\begin{corollary}
\label{cor:planes}
Fano threefolds of index~$1$ from Table~\textup{\ref{table:fanos}} contain no $H$-planes over~$\bkk$.
\end{corollary}

\begin{proof}
If~$\Pi \subset X_\bkk$ is an $H$-plane, the 
restriction~$(H_1 + \dots + H_r)\vert_\Pi$ is the line class.
Since~all the~$H_i$
are nef, it follows 
that~$H_j\vert_\Pi \sim 0$ for all $j \ne i$ and some~$i$,
hence~$\Pi$ is contracted to a point by the projection~$\pi_\whi$.
It remains to note that the fibers of~$\pi_\whi$ are at most 1-dimensional by Lemma~\ref{lemma:pi-i}.
\end{proof}

\subsection{Lines}
\label{subsec:lines}

By a \textsf{line on~$X$} we understand a curve (defined over~$\bkk$) of anticanonical degree~$1$.
We denote by~$\rF_1(X)$ the Hilbert scheme of lines on~$X$.
Note that $\rF_1(X)_\bkk \cong \rF_1(X_\bkk)$. 

\begin{lemma}
\label{lemma:lines-1}
Let $X$ be a Fano threefold of types~\type{2,2,2}, \type{4,4}, \type{3,3}, or~\type{1,1,1,1}.
A line on~$X$ is a fiber of the exceptional divisor of one of the projections~\eqref{eq:pi-hi}.
In particular
\begin{equation*}
\rF_1(X_\bkk) \cong \bigsqcup_{i=1}^r\ \Gamma_\whi,
\end{equation*}
where the smooth curves~$\Gamma_\whi$ have been described in Lemma~\xref{lemma:pi-i}.
The normal bundle of each line is
\begin{equation}
\label{eq:normal-lines}
\cN_{L/X_\bkk} \cong \cO_L \oplus \cO_L(-1).
\end{equation}
Finally,
the action of the Galois group~$\Gal(\bkk/\kk)$ on the set of connected components of the Hilbert scheme of lines 
factors through the group~$\rG_X$ and is transitive.
\end{lemma}
\begin{proof}
Since the classes $H_i$ are semiample, it follows from~\eqref{eq:kx} that for 
each $\bkk$-line~$L$ on~$X$ 
there is a unique~$i$ such that $L \cdot H_i = 1$ and $L \cdot H_j = 0$ for $j \ne i$ 
(i.e., $[L] = \Lambda_i$ in the notation of Lemma~\ref{lemma:picard}).
Thus, $L$ is contracted by the projection~$\pi_\whi$, hence it is equal to a fiber of the exceptional divisor of this projection.
Taking into account the description of the projections~$\pi_\whi$ from Lemma~\ref{lemma:pi-i}, we obtain the description of~$\rF_1(X_\bkk)$.

Further, the description of the normal bundle of~$L$ follows from the exact sequence
\begin{equation*}
0 \longrightarrow \cN_{L/E_\whi} \longrightarrow \cN_{L/X_\bkk} \longrightarrow \cN_{E_\whi/X_\bkk}\vert_L \longrightarrow 0,
\end{equation*}
because the first term is trivial and the last is~$\cO_L(-1)$.
Finally, factorization of the Galois action on the set of connected components 
of~$\rF_1(X_\bkk)$ and its transitivity follow from Lemma~\ref{lemma:picard}.
\end{proof}

For a $\bkk$-point~$x \in X$ we denote 
by~$\rF_1(X_\bkk,x) \subset \rF_1(X_\bkk)$ the 
subscheme parameterizing lines passing through~$x$.
We will need the following observation.

\begin{lemma}
\label{lemma:lines}
Let $X$ be a Fano threefold of 
types~\type{2,2,2}, \type{4,4}, \type{3,3}, or~\type{1,1,1,1}.
If~$x \in X(\bkk)$, the scheme~$\rF_1(X_\bkk,x)$ is a finite reduced scheme of length at most~$r = \uprho(X_\bkk)$.
If, moreover, $x \in X(\kk)$ then either~$\rF_1(X_\bkk,x) = \varnothing$, 
or~$\rF_1(X_\bkk,x)$ is a reduced scheme of length~$r$ 
and the Galois group~$\Gal(\bkk/\kk)$ action on~$\rF_1(X_\bkk,x)$ 
factors through the group~$\rG_X$ and is transitive.
\end{lemma}
\begin{proof}
By Lemma~\ref{lemma:lines-1} for each $\bkk$-point~$x$ of~$X$ there is at 
most one line 
from each of the connected components of the Hilbert scheme $\rF_1(X_\bkk)$ 
passing through~$x$.
This proves that~$\rF_1(X_\bkk,x)$ is finite and reduced and gives the bound for 
its length.

Now assume $x$ is a point defined over~$\kk$ and let $L$ be a $\bkk$-line 
through~$x$.
Then for any~$g \in \Gal(\bkk/\kk)$ the line $g(L)$ also passes through~$x$.
Transitivity of the Galois action on the set of components of~$\rF_1(X_\bkk)$
then implies that there is a unique line of each type through~$x$, hence the 
length of~$\rF_1(X_{\bkk},x)$ is~$r$,
and the $\Gal(\bkk/\kk)$-action on~$\rF_1(X_\bkk,x)$ factors 
through~$\rG_X$ and is transitive.
\end{proof}

\subsection{Conics}
\label{subsec:conics}

By a \textsf{conic on~$X$} we understand a connected curve (defined over~$\bkk$) of anticanonical degree~$2$.
We denote by~$\rF_2(X)$ the Hilbert scheme of conics on~$X$.
Note that $\rF_2(X)_\bkk \cong \rF_2(X_\bkk)$. 

\begin{lemma}
\label{lemma:conics}
Let $X$ be a Fano threefold of types~\type{2,2,2}, \type{4,4}, \type{3,3}, 
or~\type{1,1,1,1}.
We have the following descriptions of the Hilbert schemes of 
conics~$\rF_2(X_{\bkk})$:
\begin{eqnarray*}
\rF_2((\mtype{2,2,2})_\bkk) &\cong& \PP^2_\bkk \sqcup \PP^2_\bkk \sqcup \PP^2_\bkk,\\
\rF_2((\mtype{4,4})_\bkk) &\cong& \Gamma_1 \times \Gamma_2,\\
\rF_2((\mtype{3,3})_\bkk) &\cong& \Sym^2\Gamma_1 \cong \Sym^2\Gamma_2,\\
\rF_2((\mtype{1,1,1,1})_\bkk) &\cong& \bigsqcup_6 \ (\PP^1_\bkk \times \PP^1_\bkk) ,
\end{eqnarray*}
where $\Gamma_i$ are the curves described in Lemma~\xref{lemma:pi-i}.

Moreover, the morphism from each component of the universal conic to~$X_\bkk$ is dominant.
\end{lemma} 

\begin{proof}
First, note that no conic on~$X$ is contracted by the projections~$\pi_\whi$,
since by~Lemma~\xref{lemma:pi-i} any reduced connected curve contracted by~$\pi_\whi$ is a line 
and lines do not support nonreduced conics by~\eqref{eq:normal-lines} and~\cite[Remark~2.1.7]{KPS}.
Therefore, we deduce from~\eqref{eq:kx} that for each $\bkk$-conic~$C \subset X_{\bkk}$ 
there is a pair of indices~$1 \le i_1 < i_2 \le r$ such that 
\begin{equation}
\label{eq:intersections-conic}
H_{i_1} \cdot C = H_{i_2} \cdot C = 1
\qquad\text{and}\qquad 
H_j \cdot C = 0\quad\text{for $j \not\in \{i_1,\,i_2\}$}.
\end{equation}
If~$r \ge 3$, i.e., if~$X$ is of type~\type{2,2,2} or~\type{1,1,1,1}, such~$C$ is contracted by one of the projections
\begin{equation}
\label{eq:conic-bundles}
\pi_i \colon X_\bkk \longrightarrow \PP^2_\bkk
\qquad\text{or}\qquad 
\pi_{i_1,i_2} \colon X_\bkk \longrightarrow \PP^1_\bkk \times \PP^1_\bkk,
\end{equation}
respectively.
It is easy to see that the maps~\eqref{eq:conic-bundles} are flat conic bundles, hence~$C$ is a fiber of one of them,
and therefore~$\rF_2(X_\bkk)$ is the disjoint union of~$\PP^2_\bkk$, or of~$\PP^1_\bkk \times \PP^1_\bkk$, respectively. 

Assume~$X$ is of type~\type{4,4}.
Applying Corollary~\ref{cor:f2x-gamma1} twice we obtain a morphism
\begin{equation*}
\varphi = (\varphi_1, \varphi_2) \colon \rF_2(X_\bkk) \longrightarrow \Gamma_1 \times \Gamma_2
\end{equation*}
that takes a smooth conic~$C \subset X_\bkk$ to the unique pair of lines~$(L_2,L_1)$ of different types such that~$C \cap L_i \ne \varnothing$.
We will show that~$\varphi$ is an isomorphism.

First, note that by~\eqref{eq:intersections-conic} if~$C \subset X$ is a conic then~$\pi_1(C) \subset Q_1$ and~$\pi_2(C) \subset Q_2$ are lines,
and by Lemma~\ref{lemma:pi-i}\ref{lemma:pi-i:44} they intersect the curves~$\Gamma_1$ and~$\Gamma_2$, respectively.
Thus, by Corollary~\ref{cor:f2x-gamma1} for~$x_1 \in \Gamma_1$ if~$[C] \in \varphi_1^{-1}(x_1)$ the line~$\pi_1(C) \subset Q_1$ passes through~$x_1$.
Since any line on~$Q_1$ through~$x_1$ lies in the embedded tangent space to~$Q_1$ at~$x_1$,
and the intersection of this tangent space with~$Q_1$ is a 2-dimensional quadratic cone with vertex at~$x_1$,
it follows that
\begin{equation*}
\varphi_1^{-1}(x_1) \cong \PP^1
\end{equation*}
for any~$x_1 \in \Gamma_1$.
Since also~$\Gamma_2 \cong \PP^1$, the morphism~$\varphi$ is a morphism of~$\PP^1$-bundles over~$\Gamma_1$,
and to show that it is an isomorphism, it is enough to check that it is birational. 

So, consider a general pair~$(L_2,L_1)$ of lines on~$X$ of different types.
It follows from~\eqref{eq:intersections-conic} and Lemma~\ref{lemma:pi-i}\ref{lemma:pi-i:44} 
that~$\bar{L}_1 := \pi_1(L_1)$ is a line on~$Q_1$ bisecant to~$\Gamma_1$, $x_1 := \pi_1(L_2)$ is a point on~$\Gamma_1$, 
and by Corollary~\ref{cor:f2x-gamma1} the preimage~$\varphi^{-1}(L_2,L_1)$ 
is the Hilbert scheme of lines~$L \subset Q_1$ passing through~$x_1$ and intersecting~$\bar{L}_1$.
By genericity we may assume~$x_1 \not\in \bar{L}_1$ (i.e., that the lines~$L_1$ and~$L_2$ do not intersect).
Then any line~$L$ as above is contained in the intersection of the plane spanned by~$\bar{L}_1$ and~$x_1$ with~$Q_1$, 
which is equal to the union of the line~$\bar{L}_1$ with a residual line.
Therefore, $L$ must be equal to the residual line, hence the scheme~$\varphi^{-1}(L_2,L_1)$ consists of a single point,
so~$\varphi$ is birational, and hence it is an isomorphism.

Since the embedded tangent space to~$Q_1$ at a general point~$x \in Q_1$ intersects the quartic curve~$\Gamma_1$ at~$4$ points,
there are~$4$ lines on~$Q_1$ through~$x$ intersecting~$\Gamma_1$, 
hence the universal conic is dominant of degree~$4$ over~$X_\bkk$.

Finally, assume $X$ is of type~\type{3,3}.
By~\eqref{eq:intersections-conic} and Lemma~\ref{lemma:pi-i}\ref{lemma:pi-i:33}
the image of~$C$ with respect to the blowup~$\pi_i \colon X_\bkk \to \PP^3$ is a line
intersecting the curve~$\Gamma_{i} \subset \PP^3$ at two points.
This defines a morphism 
\begin{equation*}
\rF_2(X_\bkk) \longrightarrow \Sym^2\Gamma_{i},
\end{equation*}
and it is easy to see that it is an isomorphism.
It is also easy to see that for a general point~$x \in \PP^3$ there are seven lines passing through~$x$ and bisecant to~$\Gamma_1$;
therefore the universal conic on~$X_\bkk$ is dominant of degree~$7$ over~$X_\bkk$.
\end{proof}

\begin{remark}
\label{rem:f2x-theta}
Let~$X$ be a threefold of type~\type{4,4}.
Clearly a general line on the quadric~$Q_1$ passing through a point~$x \in \Gamma_1$ is not bisecant to~$\Gamma_1$
and its strict transform in~$X$ intersects the line~$L_2 = \pi_1^{-1}(x)$ transversally.
This means that a general conic intersecting~$L_2$ is smooth and intersects~$L_2$ transversally.
\end{remark} 

For a given curve~$\Theta \subset X$ we denote by~$\rF_2(X,\Theta)$ 
the subscheme of the Hilbert scheme~$\rF_2(X)$
that parameterizes conics intersecting the curve~$\Theta$
and by~$\mathscr{C}_2(X,\Theta) \subset \rF_2(X,\Theta) \times X$ the restriction of the universal family of conics.

\begin{lemma}
\label{lemma:conics-c0}
If $X$ is of type~\type{4,4} and $\Theta$ is a singular conic then~$\rF_2(X_\bkk,\Theta) \cong \Gamma_1 \cup \Gamma_2$
is the union of the two rulings of the surface~$\rF_2(X_\bkk) \cong \Gamma_1 \times \Gamma_2$.
Moreover, the natural projection~$\mathscr{C}_2(X,\Theta) \to X$ is birational onto
an anticanonical divisor~$R_{\Theta} \subset X$ passing through each component of the curve~$\Theta$ with multiplicity~$3$.
\end{lemma}

\begin{proof}
Let $L_1$ and $L_2$ be the irreducible components (over~$\bkk$) of the conic~$\Theta$.
The argument of Lemma~\ref{lemma:conics} shows that~$L_i$ are lines of two different types and
\begin{equation*}
\rF_2(X_\bkk,\Theta) = \rF_2(X_\bkk,L_1) \cup \rF_2(X_\bkk,L_2).
\end{equation*}
Recall that by Lemma~\ref{lemma:lines-1} the curves~$\Gamma_1$ 
and~$\Gamma_2$ 
can be identified with the two connected components of~$\rF_1(X_\bkk)$ 
and the isomorphism $\rF_2(X_\bkk) \cong \Gamma_1 \times \Gamma_2$ of 
Lemma~\ref{lemma:conics}
is defined by taking a conic~$C$ to the unique pair of lines of different 
types intersecting~$C$.
This means that
\begin{equation*}
\rF_2(X_\bkk,\Theta) \cong 
(\Gamma_1 \times [L_1]) \cup ([L_2] \times \Gamma_2) 
\subset \Gamma_1 \times \Gamma_2;
\end{equation*}
thus~$\rF_2(X_\bkk,\Theta)$ is the union of two rulings of~$\rF_2(X)$
and~$\mathscr{C}_2(X_\bkk,\Theta) = \mathscr{C}_2(X_\bkk,L_1) \cup \mathscr{C}_2(X_\bkk,L_2)$.

Furthermore, it follows from the description of Lemma~\ref{lemma:conics} that 
the morphism~$\mathscr{C}_2(X_\bkk,L_2) \to X_\bkk$ is birational onto
the hyperplane section tangent to~$Q_1$ at the point~$\pi_1(L_2)$;
it contains the line~$\pi_1(L_1)$ with multiplicity~$1$ and 
has multiplicity~$2$ at the point~$\pi_1(L_2)$.
Similarly, the morphism~$\mathscr{C}_2(X_\bkk,L_1) \to X_\bkk$ is birational onto
the hyperplane section containing the line~$\pi_2(L_2)$ with multiplicity~$1$ 
and having multiplicity~$2$ at the point~$\pi_2(L_1)$.
Thus, the morphism~\mbox{$\mathscr{C}_2(X,\Theta) \to X$} is birational onto a divisor of
class $(H_1 - L_1 - 2L_2) + (H_2 - L_2 - 2L_1) = H - 3\,\Theta$.
\end{proof}

\subsection{Twisted cubic curves}

Finally, we describe the Hilbert scheme~$\rF_3(X,x)$ of subschemes of~$X$ 
with Hilbert polynomial~$3t + 1$ with respect to~$H$ that pass through a point~$x$;
since~$X$ is an intersection of quadrics (Theorem~\ref{th:bht}) and contains no planes (Corollary~\ref{cor:planes}) 
every such subscheme is a union of rational curves (see~\cite[Lemma~2.9]{KP19}), 
so we will use the name {\sf rational normal cubic curves} for subschemes parameterized by~$\rF_3(X,x)$.
We denote by~$\mathscr{C}_3(X,x) \subset \rF_3(X,x) \times X$ the restriction of the universal family of curves.
Recall the curves~$\Gamma_\whi$ described in Lemma~\xref{lemma:pi-i}.

\begin{lemma}
\label{lemma:cubics-x}
Let $X$ be a Fano threefold of types~\type{2,2,2}, \type{4,4}, \type{3,3}, 
or~\type{1,1,1,1}.
If~$x$ is a $\kk$-point on~$X$ not lying on a $\bkk$-line, 
one has the following descriptions of the schemes~$\rF_3(X_\bkk,x)$:
\begin{eqnarray*}
\rF_3((\mtype{2,2,2})_\bkk,x) &\cong& \Gamma_{1,2} \cong \Gamma_{1,3} \cong \Gamma_{2,3},\\ 
\rF_3((\mtype{4,4})_\bkk,x) &\cong& \PP^1_\bkk \sqcup \PP^1_\bkk,\\
\rF_3((\mtype{3,3})_\bkk,x) &\cong& \Gamma_1 \sqcup \Gamma_2,\\
\rF_3((\mtype{1,1,1,1})_\bkk,x) &\cong& \bigsqcup_8 \PP^1_\bkk.
\end{eqnarray*}
Moreover, for threefolds of type~\type{4,4} 
the natural projection~$\mathscr{C}_3(X,x) \to X$ is birational onto
an anticanonical divisor~$R_x \subset X$ passing through the point~$x$ with multiplicity~$4$.
\end{lemma}
\begin{proof}
First, consider a threefold~$X$ of type~\type{2,2,2}.
If~$C$ is a rational normal cubic curve and~\mbox{$H_i \cdot C = 0$} for some~$i$ 
then $C$ is contracted by one of the conic bundles~\eqref{eq:conic-bundles},
hence~$C$ is supported on a fiber of~\eqref{eq:conic-bundles}.
But the conormal bundle of any such fiber is trivial, hence it cannot support a nonreduced 
curve of arithmetic genus~$0$ and degree more than~$2$.
This means that we have~$H_i \cdot C = 1$ for each~$i$, and we conclude from this and Lemma~\xref{lemma:pi-i}
that the image of~$C$ under the map~\mbox{$\pi_{1,2} \colon X_\bkk \to W_{1,2}$}
is a rational curve of bidegree~$(1,1)$ intersecting the curve~$\Gamma_{1,2}$ and passing through~$x$.
The argument analogous to that of Corollary~\ref{cor:f2x-gamma1} shows that there is a morphism
\begin{equation*}
\varphi_{1,2} \colon \rF_3(X,x) \longrightarrow \Gamma_{1,2}
\end{equation*}
that takes a twisted cubic curve~$C$ to the unique point~$x_{1,2} \in \Gamma_{1,2}$ such that~$C \cap \pi_{1,2}^{-1}(x_{1,2}) \ne \varnothing$.
This morphism is an isomorphism, because
on~$W_{1,2}$ there is a unique curve of bidegree~$(1,1)$ through a given pair of points 
(unless they lie on a fiber of either of the projections~$W_{1,2} \to \PP^2_\bkk$, in which case~$x$ lies on a line in~$X$).
The same argument proves isomorphisms of~$\rF_3(X,x)$ with the curves~$\Gamma_{1,3}$ and~$\Gamma_{2,3}$. 

Next, consider a threefold of type~\type{4,4}.
If $H_i \cdot C = 0$ for some~$i$ then $C$ is contracted by~$\pi_i$, hence is supported on a line. 
But the conormal bundle of a line is globally generated by~\eqref{eq:normal-lines}, 
hence a line cannot support a nonreduced 
curve of arithmetic genus~$0$ and degree more than~$1$.
This means that~$C$ has bidegree~$(1,2)$ or~$(2,1)$.
In the first case the image of~$C$ under~$\pi_1$ is a line on the quadric~$Q_1$ passing through~$x$;
hence the corresponding component of~$\rF_3(X,x)_\bkk$ is isomorphic to~$\PP^1_\bkk$.
It also follows that the corresponding component of~$\mathscr{C}_3(X,x)$ is a Hirzebruch surface that maps birationally onto
the hyperplane section of~$Q_1$ tangent at~$x$, i.e., a 
divisor of class~$H_1$ passing through~$x$ with multiplicity~$2$.
The second component is described analogously.
The total divisor class of the image~$R_x$ of~$\mathscr{C}_3(X,x) \to X$
is~$H_1 + H_2 - 4x$, i.e., it is the anticanonical class passing through~$x$ with multiplicity~$4$.

Next, consider a threefold of type~\type{3,3}.
The same argument as above shows that~$C$ has bidegree~$(1,2)$ or~$(2,1)$.
In the first case the image of~$C$ under~$\pi_1$ is a line on~$\PP^3_\bkk$ passing through~$x$ and intersecting the curve~$\Gamma_1$.
Since for any point of~$\Gamma_1$ there is a unique line through it and~$x$, 
the corresponding component of~$\rF_3(X_\bkk,x)$ is isomorphic to~$\Gamma_1$.
Analogously, the second component is isomorphic to~$\Gamma_2$.

Finally, consider a threefold of type~\type{1,1,1,1}.
Then, of course, $H_i \cdot C = 0$ for some~$i$.
The argument used for threefolds of type~\type{2,2,2} shows this cannot hold for two distinct~$i$.
So, assume this holds for~$i = 1$.
By Lemma~\ref{lemma:pi-i}\ref{lemma:pi-i:x1111} the image of~$C$ under the map~$\pi_{1,2,3}$ 
is a curve of multidegree~$(0,1,1)$ on~$\PP^1_\bkk \times \PP^1_\bkk \times \PP^1_\bkk$
intersecting the curve~$\Gamma_{1,2,3}$ and passing through~$x$.
In other words, it is a curve of bidegree~$(1,1)$ on the surface~$\PP^1_\bkk \times \PP^1_\bkk$ passing through~$x$ 
and either of the two points of intersection of~$\Gamma_{1,2,3}$ with this surface
(note that these points cannot collide because otherwise~$x$ would lie on a line in~$X$).
Therefore, there are two pencils of such curves.
Using the same argument for other~$i$ we see that altogether there are~8 
pencils of twisted cubic curves on~$X$ passing through~$x$.
\end{proof}

\begin{remark}
\label{rem:f3x-x}
Let~$X$ be a threefold of type~\type{4,4}.
Since a general line on~$Q_1$ passing through a point~$x \not\in \Gamma_1$ does 
not intersect~$\Gamma_1$,
it follows that a general twisted cubic curve on~$X$ passing through~$x$ is 
smooth.
\end{remark} 

\section{A birational transformation for a product of projective spaces}
\label{sec:toric} 

In this section we construct a birational transformation for a product of projective spaces 
and deduce a consequence for the rationality of its $\kk$-forms;
in particular we prove the rationality criterion for threefolds~\type{1,1,1}.

\subsection{Product of projective spaces}
\label{subsec:ppp}

Consider the product 
\begin{equation*}
Y = \PP^{n_1} \times \PP^{n_2} \times \dots \times \PP^{n_r} = \PP(V_1) \times 
\PP(V_2) \times \dots \times \PP(V_r)
\end{equation*}
of projective spaces.
Assume that $r = p + q$ and
\begin{equation}
\label{eq:n-i}
n_1 \ge n_2 \ge \dots \ge n_p \ge 2,
\qquad 
n_{p+1} = \dots = n_{p+q} = 1.
\end{equation}
Let $y \in Y$ be a point, and let $(v_1,v_2,\dots,v_r)$, $0 \ne v_i \in V_i$, be the corresponding collection of vectors.
Consider the blowup
\begin{equation*}
\tY = \Bl_y(Y)
\end{equation*}
and let $E \subset \tY$ be its exceptional divisor.

Let~$\PGL(V_i)_{v_i} \subset \PGL(V_i)$ be the stabilizer of the point~$[v_i] \in \PP(V_i)$ in the projective linear group $ \PGL(V_i)$.
The group 
\begin{equation*}
G = \prod_{i=1}^r \PGL(V_i)_{v_i}
\end{equation*}
acts naturally on~$\tY$ and has 
finitely many orbits, which can be described as follows.
First, for each $1 \le i \le r$ let
\begin{equation}
\label{eq:divisor-ei}
\tY_i := \Bl_y\Big( \PP(V_1) \times \dots \times \PP(V_{i-1}) \times [v_i] \times \PP(V_{i+1}) \times \dots \times \PP(V_r)\Big) \subset \tY.
\end{equation}
Furthermore, for any subset $I \subsetneq\{1,\dots,r\}$ denote
\begin{equation}
\label{eq:eI}
\tY_I := \bigcap_{i \in I} \tY_i
\qquad\text{and}\qquad 
E_I := E \cap \tY_I.
\end{equation} 
Finally, set 
\begin{equation}
\label{eq:eI-circ}
\tY_I^\circ := \tY_I \setminus \left( E_I \cup \bigcup_{I \subsetneq J} 
\tY_J \right)
\qquad\text{and}\qquad 
E_I^\circ := E_I \setminus \left( \bigcup_{I \subsetneq J} E_J \right).
\end{equation} 
Then $\tY_\varnothing^\circ$ is the open orbit, $E_\varnothing^\circ$ 
and~$\tY_i^\circ$, $p+1 \le i \le q$, are the orbits of codimension~$1$, 
and all other orbits have higher codimension.

To describe the other side of the transformation, denote
\begin{equation*}
\bar{V}_i := V_i / \kk v_i
\end{equation*}
and choose splittings $V_i = \kk v_i \oplus \bar{V}_i$.
They induce a direct sum decomposition 
\begin{equation*}
V_1 \otimes \dots \otimes V_{r} = 
\bigoplus_{I \subset \{1,\dots,r\}} \bar{V}_I,
\qquad\text{where}\quad 
\bar{V}_I := \bigotimes_{i \in I} \bar{V}_i.
\end{equation*}
Note that the point $y$ corresponds to the summand~$\bar{V}_\varnothing = \kk$,
and the tangent space to~$Y$ at~$y$ corresponds to the sum of the summands~$\bar{V}_I$ with~$|I| = 1$.

Note also that for $i \ge p + 1$ one has $\PP(\bar{V}_i) \cong \Spec(\kk)$.
Let 
\begin{equation*}
Y^+ := \prod_{i=1}^r \PP(\bar{V}_i) =
\PP(\bar{V}_1) \times \PP(\bar{V}_2) \times \dots \times \PP(\bar{V}_p) \cong 
\PP^{n_1-1} \times \PP^{n_2-1} \times \dots \times \PP^{n_p-1}.
\end{equation*}
Denote by~$h_i$ the pullback to~$Y^+$ of the hyperplane class of the $i$-th factor 
(note that~$h_i = 0$ for~$i \ge p + 1$ because, as we noticed above, $\PP(\bar{V}_i)$ is just a point)
and for $I \subset \{1,\dots,r\}$ set 
\begin{equation*}
h_I := \sum_{i \in I} h_i.
\end{equation*}
Consider the vector bundle~$\cE$ of rank~$r+1$ on~$Y^+$ defined by
\begin{equation}
\label{eq:ce}
\cE := \bigoplus_{|I| \ge r - 1} \cO\left( - h_I \right).
\end{equation}
Denote by
\begin{equation*}
s_i \colon Y^+ \longrightarrow \PP_{Y^+}(\cE)
\end{equation*}
the section~of~\eqref{eq:ce} corresponding to the summand with $I = \{1,\dots,i-1,i+1,\dots,r\}$.
Set 
\begin{equation}
\label{def:hyp-typ}
\hY^+ := \PP_{Y^+}(\cE),
\qquad
\tY^+ := \Bl_{s_{p+1}(Y^+) \sqcup \dots \sqcup s_{p+q}(Y^+)}(\hY^+).
\end{equation}
and let~$E_i \subset \tY^+$, $p + 1 \le i \le p + q$, be the exceptional divisors.
The group~$G$ acts transitively on~$Y^+$, the vector bundle~$\cE$ is 
$G$-equivariant, and its summands $\cO(-h_I)$ with $|I| = r - 1$ 
are~$G$-invariant.
Therefore, the action of~$G$ lifts naturally to~$\hY^+$ and~$\tY^+$.
Moreover, the action of~$G$ on~$\tY^+$ still has a finite number of orbits, 
which can be described as follows.

For a subset $J \subsetneq \{1,\dots,r\}$ denote
\begin{equation}
\label{eq:yplus-i}
\bar{\cE}_J = \bigoplus_{J \subset I,\ |I| = r - 1} \cO\left( - h_I \right);
\end{equation}
this is a subbundle in~$\cE$ of corank~$1 + |J|$.
Let~$\tY^+_J \subset \tY^+$ denote the strict transform of~$\PP_{Y^+}(\bar{\cE}_J)$.
Then the $G$-orbits are
\begin{align*}
(\tY^+)^\circ &= \tY^+ \setminus \left( \tY^+_\varnothing \cup \bigcup_{i = p + 
1}^q E_i \right),
&
(\tY^+_J)^\circ &= \tY^+_J \setminus \left(\bigcup_{i = p + 1}^q E_i \right),
\\
E_i^\circ &= E_i \setminus \tY^+_\varnothing,
&
E_{i,J}^\circ &= (E_i \cap \tY^+_J) \setminus \left( \bigcup_{J \subsetneq K} E_i \cap \tY^+_K \right),
\end{align*}
where in the last formula we assume~$i \not\in J$.
Note that~$(\tY^+)^\circ$ is the open orbit, 
$(\tY^+_\varnothing)^\circ$ and~$E_i^\circ$ are the orbits of codimension~$1$, 
and all other orbits have higher codimension. 

The linear projection out of~$[v_i]$ defines a~$\PGL(V_i)_{v_i}$-equivariant rational map $\PP(V_i) 
\dashrightarrow \PP(\bar{V}_i)$ which is regular if $i \ge p + 1$.
The product of these maps is a $G$-equivariant rational map, which we denote by~$\psi_0 \colon Y \dashrightarrow Y^+$.
It gives rise to the following birational transformation.

\begin{theorem}
\label{proposition:toric-link}
There is a small birational $G$-equivariant isomorphism~$\psi \colon \tY 
\dashrightarrow \tY^+$ that fits into the commutative diagram
\begin{equation}
\label{eq:sl-ppp}
\vcenter{
\xymatrix@C=5em{
\tY \ar[dd]_{\sigma} \ar@{-->}^{\psi}[r] \ar@{..>}[dr]_{\hat\psi} &
\tY^+ \ar^{\tilde\sigma_+}[d] \ar@/^3em/[dd]^{\sigma_+}
\\
&
\hY^+ \ar[d]^{\hat\sigma_+} 
\\
Y \ar@{-->}[r]^{\psi_0} \ar@{..>}[ur]^{\hat\psi_0} &
Y^+,
} }
\end{equation} 
where~$\hat\sigma_+ \colon \hY^+ = \PP_{Y^+}(\cE) \to Y^+$ and
\mbox{$\tilde\sigma_+ \colon \tY^+ = \Bl_{s_{p+1}(Y^+) \sqcup \dots \sqcup s_{p+q}(Y^+)}(\hY^+) \to \hY^+$} is the projection and  the blowup, respectively,
$\sigma_+ := \hat\sigma_+ \circ \tilde\sigma_+$
and, such that~$\psi$ induces isomorphisms of~$G$-orbits 
\begin{equation*}
\tY_\varnothing^\circ \cong (\tY^+)^\circ,
\qquad 
E_\varnothing^\circ \cong (\tY^+_\varnothing)^\circ,
\qquad\text{and}\qquad 
\tY_i^\circ \cong E_i^\circ
\end{equation*}
of codimension~$0$ and~$1$.
Moreover, if 
\begin{itemize}
\item 
$H_i$, $1 \le i \le r$, are the hyperplane classes of $\PP(V_i)$
and
$H = H_1 + \dots + H_r$, 
\item 
$E$ is the exceptional divisor of $\sigma$, 
\item 
$h$ is the relative hyperplane class of the projective bundle~$\hat\sigma_+$, 
\item 
$h_i$, $1 \le i \le p$, are the hyperplane classes of~$\PP(\bar{V}_i)$, and
\item 
$e_i$, $p+1 \le i \le p+q$, are the exceptional divisor classes of the blowup~$\tilde\sigma_+$,
\end{itemize}
then in the Picard group $\Pic(\tY) = \Pic(\tY^+)$ there are the following 
equalities 
\begin{equation}
\label{eq:toric-pic-1}
\begin{array}{lclll}
h_i &=& H_i - E,\quad && 1 \le i \le p,\\
e_i &=& H_i - E,\quad && p+1 \le i \le p+q,\\
h &=& H - (r - 1)E.\qquad&&
\end{array}
\end{equation}

Conversely, one has
\begin{equation}
\label{eq:toric-pic-e}
E = h - \sum_{i=1}^p h_i - \sum_{j=p+1}^{p+q} e_j,
\qquad 
H_i = 
\begin{cases}
h_i + E, & \text{for $1 \le i \le p$},\\
e_i + E, & \text{for $p + 1 \le i \le p + q$}.
\end{cases}
\end{equation} 
\end{theorem}

The maps~$\hat\psi_0$ and~$\hat\psi$ in~\eqref{eq:sl-ppp} will be defined in the course of proof.

\begin{proof}
For each $u_i \in V_i$ denote by~$\bar{u}_i \in \bar{V}_i$ the image of~$u_i$ under the linear projection from the fixed vector~$v_i \in V_i$.
Then the rational map~$\psi_0 \colon Y \dashrightarrow Y^+$ is given by the 
formula 
\begin{equation*}
(u_1,\dots,u_r) \longmapsto (\bar{u}_1,\dots,\bar{u}_r).
\end{equation*}
This map is regular on the open orbit~$Y^\circ \subset Y$ (given by the 
conditions~$\bar{u}_i \ne 0$ for all~\mbox{$1 \le i \le r$})
and it extends regularly to the orbits~$Y_i^\circ \subset Y$ of codimension~$1$ 
(given by the condition~$\bar{u}_i = 0$ for some~$p + 1 \le i \le p+ q$ 
and~$\bar{u}_j \ne 0$ for all $j \ne i$).

Now consider the rational $G$-equivariant map
\begin{equation}
\label{eq:tpsi-1}
\hat\psi_0 \colon Y \dashrightarrow \hY^+,
\qquad 
(u_1,\dots,u_r) \longmapsto \left((\bar{u}_1,\dots,\bar{u}_r), \sum_{|I| \ge r - 1} 
\bigotimes_{i \in I} \bar{u}_i\right).
\end{equation}
Here we consider the summand~$\otimes_{i \in I} \bar{u}_i$ 
as a point in the fiber of the line bundle~$\cO(-h_I)$ and their sum for $|I| 
\ge r - 1$ 
as a point in the fiber (of the projectivization) of~$\cE$. 
The map~$\hat\psi_0$ induces an isomorphism of the open orbit~$Y^\circ \subset 
Y$ 
onto the open orbit~$\PP_{Y^+}(\cE) \setminus \PP_{Y^+}(\bar{\cE}_{\varnothing})$ in~$\hY^+$
and contracts each orbit~$Y_i^\circ$ of codimension~$1$ to the section~$s_i(Y^+) \subset \hY^+$, $p + 1 \le i \le p + q$.

Now consider the composition~$\hat\psi = \hat\psi_0 \circ \sigma \colon \tY \dashrightarrow \hY^+$. 
The restriction of~$\hat\psi$ to the exceptional divisor~$E$ is given by
\begin{equation}
\label{eq:tpsi-2}
E = \PP(\bar{V}_1 \oplus \dots \oplus \bar{V}_r) \dashrightarrow \hY^+, 
\qquad 
(\bar{u}_1 + \dots + \bar{u}_r) \longmapsto \left((\bar{u}_1,\dots,\bar{u}_r), 
\sum_{|I| = r - 1} \bigotimes_{i \in I} \bar{u}_i\right).
\end{equation}
It maps the orbit~$E_\varnothing^\circ \subset \tY$ isomorphically 
onto the $G$-orbit~$(\tY^+_{\varnothing})^\circ = \PP_{Y^+}(\bar{\cE}) \setminus \left( \bigcup_{i=1}^r 
\PP_{Y^+}(\bar{\cE}_i) \right)$ of codimension~$1$.
By the above arguments it also gives an isomorphism of open $G$-orbits 
and contracts the orbits~$\tY_i^\circ \cong Y_i^\circ$, $p + 1 \le i \le p + q$,
to the sections~$s_i(Y^+) \subset \hY^+$.
Therefore, $\hat\psi$ induces a birational isomorphism
\begin{equation*}
\psi \colon \tY \dashrightarrow 
\Bl_{s_{p+1}(Y^+) \sqcup \dots \sqcup s_{p+q}(Y^+)}(\PP_{Y^+}(\cE)) = \tY^+.
\end{equation*}
Finally, it is easy to see that the induced map~$\tY_i^\circ \to E_i^\circ$ is an isomorphism for all~$p + 1 \le i \le p + q$.
This gives the commutative diagram~\eqref{eq:sl-ppp} and proves that~$\psi$ is 
small. 

The first two lines in~\eqref{eq:toric-pic-1} follow easily from the 
formulas~\eqref{eq:tpsi-1}, \eqref{eq:tpsi-2}, and~\eqref{eq:divisor-ei}.
The last line follows from the equality of the canonical classes of~$\tY$ 
and~$\tY^+$ 
expressed in terms of~$H_i$ and~$E$ on the one hand,
and~$h_i$, $h$, and~$e_i$ on the other hand.

Finally,~\eqref{eq:toric-pic-e} follows from~\eqref{eq:toric-pic-1}.
\end{proof}

\begin{remark}
\label{rem:flips}
Alternatively, one can use the fact that the varieties~$\tY$ and~$\tY^+$, as 
well as the birational isomorphism~$\psi$ are toric.
Thus, to check that~$\psi$ is small, it is enough to identify the generators of 
rays of the corresponding fans.
Moreover, comparing the other cones in the fans one can check that the 
map~$\psi$ factors as the composition
\begin{equation*}
\xymatrix@1@C=4em{
\tY\ \ar@{-->}[r]^{\psi_1} & 
\ \tY'\ \ar@{-->}[r]^{\psi_2} &
\ \dots\ \ar@{-->}^{\psi_{r-2}}[r] &
\ \tY^{(r-2)}\ \ar@{-->}^{\psi_{r-1}}[r] &
\ \tY^+
}
\end{equation*}
of standard (anti)flips~$\psi_l$ in the strict transforms of~$\tY_I$
for~\mbox{$|I| = l$}, $1 \le l \le r - 2$, and 
for~\mbox{$|I| = r - 1$} with~$\{p+1,\dots,p+q\} \subset I$, respectively.
\end{remark} 

\subsection{Rationality of forms of products of projective spaces}

Here we apply the birational transformation of the previous subsection to 
deduce the following
corollary (see~\cite{Zak07} for a different proof).

\begin{corollary}
\label{corollary:product-rational}
Let $Y$ be a $\kk$-form of $\PP^{n_1} \times \PP^{n_2} \times \dots \times \PP^{n_r}$.
For any~$y \in Y(\kk)$ the diagram~\eqref{eq:sl-ppp} is defined over~$\kk$
and if~$Y(\kk) \ne \varnothing$ then~$Y$ is~$\kk$-rational.
\end{corollary}

\begin{proof}
First, we prove that for any~$y \in Y(\kk)$ the diagram~\eqref{eq:sl-ppp} is defined over~$\kk$.
The divisor classes~$H = \sum_{i=1}^r H_i$ and~\mbox{$H' := \sum_{i=1}^p H_i$} 
on~$Y_\bkk$ are Galois-invariant, 
and since we have~\mbox{$Y(\kk) \ne \varnothing$} by assumption, we conclude 
that they are defined over~$\kk$.
Also~$\tY$ and~$E$ are defined over~$\kk$ as~$y$ is a $\kk$-point.
Therefore, the divisor classes
\begin{equation*}
\sum_{i=1}^p h_i = H' - pE,
\qquad 
h = H - (r - 1)E,
\qquad\text{and}\qquad 
-\sum_{i=p+1}^{p+q} e_i = H - H' - qE
\end{equation*}
(which by Theorem~\ref{proposition:toric-link} are equal to the strict transforms of the classes that are
ample on~$Y^+$, relatively ample for~$\hY^+ \to Y^+$ and for~$\tY^+ \to \hY^+$, respectively)
are defined over~$\kk$, hence the varieties~$Y^+$, $\hY^+$, and~$\tY^+$, equal 
to the images of~$\tY$ under the maps 
given by their appropriate linear combinations, are defined over~$\kk$, as well 
as the remaining arrows in the diagram.

Now to prove~$\kk$-rationality of~$Y$ we argue by induction in $\dim(Y) = \sum n_i$.
If the dimension is zero, there is nothing to prove.
So, assume the dimension is positive and consider the diagram~\eqref{eq:sl-ppp}.
By Theorem~\ref{proposition:toric-link} the variety~$Y^+$ 
is a $\kk$-form of $Y_\bkk^+ = \PP^{n_1-1}_\bkk \times \PP^{n_2-1}_\bkk \times \dots \times \PP^{n_r-1}_\bkk$.
By the Nishimura lemma (see \cite{nishimura-55})
we have~$Y^+(\kk) \ne \varnothing$, hence $Y^+$ is $\kk$-rational by the induction assumption.
Furthermore, the morphism~$\hY^+ \to Y^+$ is a $\kk$-form of a projective bundle, 
and by~\eqref{eq:toric-pic-e} the strict transform of the exceptional divisor~$E$ of~$\tY$ 
provides for it a relative hyperplane section. 
But~$E$ is defined over~$\kk$, therefore $\tY^+$ is rational over~$Y^+$, hence it is $\kk$-rational.
It remains to note that the morphisms~$\sigma$, $\psi$, and~$\tilde\sigma_+$ in~\eqref{eq:sl-ppp} are birational, 
hence $Y$ is $\kk$-rational as well.
\end{proof} 

Applying this to the case of a~$\kk$-form of~$(\PP^1)^3$ we obtain

\begin{corollary}
\label{cor:x111}
If $X$ is a Fano threefold of type~\type{1,1,1} with~$X(\kk) \ne \varnothing$ 
then~$X$ is $\kk$-rational.
\end{corollary}

For other applications of the theorem we will often use the following 
observation.
Recall the definitions~\eqref{eq:divisor-ei} and~\eqref{eq:yplus-i} of the 
subvarieties~$Y_i \subset Y$ of codimension~$n_i$
and subbundles~$\bar{\cE}_i \subset \cE$ of corank~$2$.

\begin{proposition}
\label{prop:x-transform}
Let~$Y$ be a~$\kk$-form of~$(\PP^n)^r$ where~$n \ge 2$ and assume~$Y$ has a 
$\kk$-point~\mbox{$y \in Y(\kk)$}.
Let~$X \subset Y$ be a closed $\kk$-subvariety containing the point~$y$ such 
that
\begin{equation*}
X_\bkk = \bigcap_{\alpha = 1}^c D_\alpha \subset (\PP^n_{\bkk})^r
\end{equation*}
is a complete intersection of divisors~$D_\alpha \subset (\PP^n_{\bkk})^r$, $1 \le 
\alpha \le c$.
Let~$\tD_\alpha \subset \tY_\bkk$ and~$\tD_\alpha^+ \subset \tY^+_\bkk = 
\hY^+_\bkk$ be the strict transforms of~$D_\alpha$ and set
\begin{equation*}
\tX^+_\bkk := \bigcap_{\alpha = 1}^c \tD_\alpha^+ \subset 
\tY^+_\bkk = \hY^+_\bkk = 
\PP_{(\PP^{n-1}_{\bkk})^r}(\cE).
\end{equation*}
If~$X$ is smooth at~$y$ and for each~$1 \le i \le r$ one has
\begin{equation}
\label{eq:dimension-conditions}
\dim ( X_\bkk \cap (Y_i)_\bkk ) < \dim(X_\bkk)
\qquad\text{and}\qquad
\dim ( \tX^+_\bkk \cap \PP_{Y^+_\bkk}(\bar{\cE_i}) ) < \dim(X_\bkk)
\end{equation}
then the strict transform~$\tX^+ = \psi_*(\Bl_y(X))$ of~$X$ in~$\tY^+$ is a 
$\kk$-form of the complete intersection~$\tX^+_\bkk$ 
and~$X$ is birational to~$\tX^+$ over~$\kk$.
\end{proposition}

\begin{proof}
By Corollary~\ref{corollary:product-rational} the diagram~\eqref{eq:sl-ppp} is defined over~$\kk$, 
so it is enough to check that the complete intersection~$\tX^+_\bkk$
is equal to the strict transform of~$\Bl_y(X_\bkk)$. 
First, note that the assumption that~$y$ is a smooth point of~$X$ implies that 
the strict transform~$\Bl_y(X_\bkk)$ of~$X_\bkk$ in~$\tY$ 
is the complete intersection of the divisors~$\tD_\alpha$.
Furthermore, the first part of the assumptions~\eqref{eq:dimension-conditions} 
implies 
that the intersection of~$\Bl_y(X_\bkk)$ with the open~$G$-orbit in~$\tY_\bkk$ 
is dense in~$\Bl_y(X_\bkk)$.
Therefore, the strict transform of~$\Bl_y(X_\bkk)$ in~$\tY^+_\bkk$ is contained in~$\tX^+_\bkk$.
So, it remains to check that~$\tX^+_\bkk$ is irreducible of dimension~$\dim(X_\bkk)$.
This is definitely true for the intersection of~$\tX^+_\bkk$ with the 
complement of the union of projective subbundles~$\PP_{Y^+_\bkk}(\bar{\cE}_i)$,
because the map~$\psi$ defines an isomorphism of this complement with an open subset of~$\tY_\bkk$.
On the other hand, the second part of the assumptions~\eqref{eq:dimension-conditions} 
gives a bound for the dimension of the intersections with these projective subbundles, which implies the irreducibility.
\end{proof} 

\section{Rationality and unirationality of types \type{2,2}, \type{2,2,2}, 
and~\type{1,1,1,1}}
\label{sec:constructions}

In this section we prove rationality of Fano threefolds of types~\type{2,2} 
and~\type{2,2,2}
as well as unirationality of threefolds of type~\type{1,1,1,1} 
under the assumption~$X(\kk) \ne \varnothing$. 

\subsection{Rationality of \type{2,2}}

To start with we deal with threefolds of type~\type{2,2}.

\begin{proposition}
\label{prop:x22}
Let $X$ be a Fano threefold of type~\type{2,2}.
If $X(\kk) \ne \varnothing$ then~$X$ is $\kk$-rational. 
\end{proposition}

\begin{proof}
Let~$x$ be a $\kk$-point of~$X$.
By definition~$X$ is a smooth divisor of bidegree~$(1,1)$ in a $\kk$-form~$Y$ of~$\PP^2 \times \PP^2$.
Since the birational isomorphism~$\psi \colon \Bl_x(Y) = \tY \dashrightarrow \tY^+ = \PP_{Y^+}(\cE)$ is small by Theorem~\ref{proposition:toric-link}, 
it follows that~$X$ is birational to a $\kk$-form of a divisor 
\begin{equation*}
\tX^+_\bkk \subset \PP_{\PP^1 \times \PP^1}(\cE) = \PP_{\PP^1 \times 
\PP^1}(\cO(-1,-1) \oplus \cO(-1,0) \oplus \cO(0,-1))
\end{equation*}
which by~\eqref{eq:toric-pic-e} has type~$H_1 + H_2 - E = h$. 
Any such divisor corresponds to a morphism
\begin{equation*}
\xi \colon \cO(-1,-1) \oplus \cO(-1,0) \oplus \cO(0,-1) \longrightarrow \cO.
\end{equation*}
Furthermore, the divisor~$\tX^+ \subset \tY^+$ comes with a morphism~$\sigma_+ \colon \tX^+ \to Y^+$ defined over~$\kk$.
By the Nishimura lemma we have~$\tX^+(\kk) \ne \varnothing$, hence~$Y^+(\kk) \ne \varnothing$, 
and since~$Y^+$ is a $\kk$-form of~$\PP^1 \times \PP^1$, it is $\kk$-rational by Corollary~\ref{corollary:product-rational}.
Finally, the general fiber of the morphism~$\sigma_+ \colon \tX^+ \to Y^+$ is a $1$-dimensional linear section of a form of a projective plane,
hence it is isomorphic to~$\PP^1$, hence~$\tX^+$ is rational over~$Y^+$, hence is~$\kk$-rational, hence so is~$X$.
\end{proof}

\begin{remark}
One can check that the birational isomorphism~$\psi \colon \tX \dashrightarrow 
\tX^+$ 
is a flop in the union of the strict transforms of two $\bkk$-lines passing 
through the point~$x \in X$
and that~$\sigma_+ \colon \tX^+ \to Y^+$ is the projectivization of the 
vector bundle~$\Ker(\xi)$ of rank~$2$ over~$Y^+$,
and that these maps provide a Sarkisov link~\eqref{eq:sl-general}.
This is an example of a pseudoisomorphism between almost del Pezzo varieties of degree~$5$,
see~\cite[Lemma~5.4 and proof of Theorem~1.2]{KP22}.
\end{remark} 

\subsection{Rationality of~\type{2,2,2}}

A similar argument works for threefolds of type~\type{2,2,2}.

\begin{proposition}
\label{prop:x222}
Let $X$ be a Fano threefold of type~\type{2,2,2}.
If $X(\kk) \ne \varnothing$ then~$X$ is $\kk$-rational. 
\end{proposition}
\begin{proof}
Let~$x$ be a $\kk$-point of~$X$.
By definition~$X_\bkk$ is a complete intersection of three divisors in~$Y_\bkk 
= \PP(V_1) \times \PP(V_2) \times \PP(V_3) \cong \PP^2 \times \PP^2 \times 
\PP^2$
of multidegree~$(1,1,0)$, $(1,0,1)$, and~$(0,1,1)$, respectively.
Denote by 
\begin{equation*}
F_{12} \in V_1^\vee \otimes V_2^\vee,
\qquad 
F_{13} \in V_1^\vee \otimes V_3^\vee,
\qquad 
F_{23} \in V_2^\vee \otimes V_3^\vee,
\end{equation*}
their equations.
We apply Proposition~\ref{prop:x-transform}; for this we consider the intersection
\begin{equation*}
\tX^+_\bkk \subset \PP_{\PP^1 \times \PP^1 \times \PP^1}(\cE) = 
\PP_{\PP^1 \times \PP^1 \times \PP^1}(\cO(-1,-1,-1) \oplus \cO(-1,-1,0) \oplus 
\cO(-1,0,-1) \oplus \cO(0,-1,-1) ) 
\end{equation*}
of the three strict transforms of the above divisors, which by~\eqref{eq:toric-pic-e} have types
\begin{equation*}
H_1 + H_2 - E = h - h_3,
\qquad 
H_1 + H_3 - E = h - h_2,
\qquad 
H_2 + H_3 - E = h - h_1,
\end{equation*}
hence correspond to a morphism of vector bundles
\begin{multline*}
\xi \colon \cO(-1,-1,-1) \oplus \cO(-1,-1,0) \oplus \cO(-1,0,-1) \oplus 
\cO(0,-1,-1)
\longrightarrow \\ \longrightarrow 
\cO(0,0,-1) \oplus \cO(0,-1,0) \oplus \cO(-1,0,0). 
\end{multline*}
It is easy to see that~$\xi$ is given by the matrix 
\begin{equation}
\label{eq:xi-x222}
\xi = 
\begin{pmatrix}
\bar{F}_{12} & 0 & F_{12}(-,v_2) & F_{12}(v_1,-) \\
\bar{F}_{13} & F_{13}(-,v_3) & 0 & F_{13}(v_1,-) \\
\bar{F}_{23} & F_{23}(-,v_3) & F_{23}(v_2,-) & 0
\end{pmatrix},
\end{equation}
where we write~$x = (v_1,v_2,v_3)$, choose splittings~$V_i = \bkk v_i \oplus \bar{V}_i$,
write $\bar{F}_{ij}$ for the 
restriction of~$F_{ij}$ to~$\bar{V}_i \otimes \bar{V}_j$,
and consider~$F_{ij}(v_i,-)$ and~$F_{ij}(-,v_j)$ as linear functions 
on~$\bar{V}_j$ and~$\bar{V}_i$ by restriction.

Let us check the dimension conditions~\eqref{eq:dimension-conditions}.
Since~$(Y_i)_\bkk$ is a fiber of the projection
\begin{equation*}
\pi_{i} \colon Y_\bkk = \PP^2 \times \PP^2 \times \PP^2 \longrightarrow \PP^2, 
\end{equation*}
it follows from Lemma~\ref{lemma:conics} that~$X_\bkk \cap (Y_i)_\bkk$ is a 
conic, hence the first part of the dimension conditions is satisfied.
To check the second part, we need to show that the restriction
\begin{equation*}
\xi_{2,3} \colon \cO(-1,0,-1) \oplus \cO(0,-1,-1) \longrightarrow \cO(0,0,-1) 
\oplus \cO(0,-1,0) \oplus \cO(-1,0,0) 
\end{equation*}
of~$\xi$ to the last two summands of~$\cE$ (given by the last two columns 
of~\eqref{eq:xi-x222}) cannot be everywhere degenerate,
and similarly for the restrictions~$\xi_{1,3}$ and~$\xi_{1,2}$.
Assuming that~$\xi_{2,3}$ is everywhere degenerate we conclude 
from~\eqref{eq:xi-x222} that 
\begin{equation*}
F_{12}(v_1,-) = F_{13}(v_1,-) = 0
\quad\text{or}\quad 
F_{12}(-,v_2) = F_{23}(v_2,-) = 0
\quad\text{or}\quad 
F_{13}(v_1,-) = F_{23}(v_2,-) = 0.
\end{equation*}
In any case it would follow that 
at least two of the bilinear forms~$F_{i,j}$ are degenerate, hence
at least two of the divisors~$W_{i,j} \subset \PP(V_i) \times \PP(V_j)$ 
(defined by the equation~$F_{i,j}$) are singular, which contradicts Lemma~\ref{lemma:pi-i}\ref{lemma:pi-i:222}.

Thus, the conditions~\eqref{eq:dimension-conditions} are satisfied, and we 
conclude from Proposition~\ref{prop:x-transform} 
that~$X$ is $\kk$-birational to a $\kk$-form of the complete 
intersection~$\tX^+_\bkk$.

Finally, the subvariety~$\tX^+ \subset \tY^+$ comes with a morphism~$\sigma_+ 
\colon \tX^+ \to Y^+$ defined over~$\kk$.
By the Nishimura lemma we have~$\tX^+(\kk) \ne \varnothing$, hence~$Y^+(\kk) \ne \varnothing$, 
and since~$Y^+$ is a $\kk$-form of~$\PP^1 \times \PP^1 \times \PP^1$, it is 
$\kk$-rational by Corollary~\ref{corollary:product-rational}.
Moreover, the general fiber of the morphism~$\sigma_+ \colon \tX^+ \to Y^+$ is 
a 0-dimensional linear section of a form of a projective space,
hence this morphism is birational, hence~$\tX^+$ is~$\kk$-rational, hence so is~$X$.
\end{proof}

\begin{remark}
If the point~$x$ does not lie on a $\bkk$-line, i.e., $\rF_1(X,x) = \varnothing$, 
one can check that the birational isomorphism~$\psi \colon \tX \dashrightarrow \tX^+$ 
is a flop in the union of the strict transforms of three smooth $\bkk$-conics 
passing through the point~$x \in X$,
that~$\sigma_+ \colon \tX^+ \to Y^+$ is the blowup of a smooth 
geometrically rational curve of multidegree~$(2,2,2)$
and that these maps provide a Sarkisov link~\eqref{eq:sl-general}.
\end{remark} 

\subsection{Unirationality of~\type{1,1,1,1}}

Finally, we deal with threefolds of type~\type{1,1,1,1}.

\begin{proposition}
\label{prop:x1111}
Let $X$ be a Fano threefold of type~\type{1,1,1,1}.
If $X(\kk) \ne \varnothing$ then~$X$ is $\kk$-unirational. 
\end{proposition}

\begin{proof}
Let~$x$ be a $\kk$-point of~$X$.
By definition~$X$ is a smooth divisor of multidegree~$(1,1,1,1)$ in 
a~$\kk$-form~$Y$ of~$\PP^1 \times \PP^1 \times \PP^1 \times \PP^1$.
The birational isomorphism~$\psi \colon \Bl_x(Y) = \tY \dashrightarrow \tY^+ = \Bl_{s_1,s_2,s_3,s_4}(\PP^4)$ 
is small by Theorem~\ref{proposition:toric-link}, 
so it follows that~$X$ is birational to a $\kk$-form of a divisor 
\begin{equation*}
\tX^+_\bkk \subset \Bl_{s_1,s_2,s_3,s_4}(\PP^4)
\end{equation*}
which by~\eqref{eq:toric-pic-e} has type~$H_1 + H_2 + H_3 + H_4 - E = 3h - 2 
\sum_{i=1}^4 e_i$, in particular, $\tX^+$ is a cubic hypersurface.
Moreover, the exceptional divisor~$E \cap \Bl_x(X) \subset \tY = \Bl_x(Y)$ of 
the blowup~$\Bl_x(X) \to X$
is an irreducible $\kk$-rational surface birational to a $\kk$-form of the 
complete intersection of~$\tX^+_\bkk$ 
with the linear span~$\PP^3 \subset \PP^4$ of the points~$s_i$. 

Let us prove that~$\tX^+_\bkk$ is not a cone.
Indeed, $\tX^+_\bkk$ is smooth away from the linear span~$\PP^3$ of the~$s_i$, 
because the map~$\psi$ from Theorem~\ref{proposition:toric-link} is an 
isomorphism over it and~$X_\bkk$ is smooth,
so if~$\tX^+_\bkk$ is a cone, its vertex belongs to the~$\PP^3$.
But then its intersection with the~$\PP^3$ (which has been shown to be an 
irreducible~$\kk$-rational surface)
is itself a cone and has a singular point at each of the~$s_i$.
But it is easy to see that any such cone is reducible; 
this contradiction proves the claim.
Now we conclude that~$\tX^+$ is~$\kk$-unirational by~\cite[Theorem~1.2]{Kollar:cubic}.
\end{proof} 

\section{Rationality of type \type{4,4}}
\label{sec:x44}

In this section we prove rationality of Fano threefolds of type \type{4,4}.

\subsection{Sarkisov links}
\label{subsec:sl-x44-existence}

We start with a construction of two Sarkisov links.
Recall that~$\rF_1(X,x)$ denotes the Hilbert scheme of lines on~$X$ passing 
through~$x$, see~\S\ref{subsec:lines};
and that by Lemma~\ref{lemma:lines} if~$x \in X(\kk)$ and~$\rF_1(X,x)$
is not empty, 
$\rF_1(X,x)$ is the union of two reduced~$\bkk$-points swapped by the Galois action.
If~$L_1$ and~$L_2$ are the corresponding~$\bkk$-lines on~$X$ (passing 
through~$x$), then 
\begin{equation*}
\Theta(x) := L_1 \cup L_2 
\end{equation*}
is a singular $\kk$-conic on~$X$ irreducible over~$\kk$ and 
with~$\Sing(\Theta(x)) = \{x\}$.

Recall that a quintic del Pezzo threefold is a Fano threefold of index~$2$ and half-anticanonical degree~$5$. 
Over an algebraically closed field it can be realized as a complete intersection
of the Grassmannian~$\Gr(2,5)$ 
with a linear subspace of codimension~$3$, 
see~\cite[Chapter~2, Theorem~1.1]{Iskovskikh-1980-Anticanonical}. 

\begin{theorem}
\label{theorem:x44-links}
Let $X$ be a Fano threefold of type~\type{4,4} and let~$x \in X(\kk)$ be a 
$\kk$-point.
\begin{enumerate}
\item 
\label{prop:sl:point-2-blowup-Q}
If\/ $\rF_1(X,x) = \varnothing$ there exists a Sarkisov link~\eqref{eq:sl-general} defined over~$\kk$, where 
\begin{itemize}
\item 
$\sigma$ is the blowup of the point~$x$,
\item 
$X^{+}$ is a smooth quintic del Pezzo threefold, and
\item 
$\sigma_+$ is the blowup of a smooth $\kk$-irreducible curve~$B^+ \subset X^+$ of degree~$4$ with two geometrically rational $\bkk$-components.
\end{itemize}
\item 
\label{prop:sl-conic}
If\/ $\rF_1(X,x) \ne \varnothing$ there exists a Sarkisov link~\eqref{eq:sl-general} defined over~$\kk$, where 
\begin{itemize}
\item 
$\sigma$ is the blowup of the singular $\kk$-irreducible 
conic~$\Theta(x)$,
\item 
$X^{+}$ is a smooth Fano threefold of type~\type{2,2}, and
\item 
$\sigma_+$ is the blowup of a singular $\kk$-irreducible curve~$B^+ \subset X^+$ of degree~$6$ with two geometrically rational $\bkk$-components.
\end{itemize}
\end{enumerate}
\end{theorem}

The proof of the theorem takes~\S\ref{subsec:sl-x44-existence} and~\S\ref{subsec:sl-x44-details}:
in the rest of~\S\ref{subsec:sl-x44-existence} we prove the existence of the links, 
and in~\S\ref{subsec:sl-x44-details} we describe them in detail.
The proofs of cases~\ref{prop:sl:point-2-blowup-Q} and~\ref{prop:sl-conic} are completely analogous,
so to carry them on simultaneously we introduce the following convenient notation
\begin{equation}
\label{eq:m-x}
m = m(x) := 
\begin{cases}
2, & \text{if $\rF_1(X,x) = \varnothing$},\\
1, & \text{if $\rF_1(X,x) \ne \varnothing$}.
\end{cases}
\end{equation}

The proof of the existence of the links is analogous to the first parts of~\cite[Theorem~5.17 and Theorem~5.9]{KP19},
so we use below some results from~\cite[\S5.1]{KP19}.

Let 
\begin{equation*}
\sigma \colon \tX \longrightarrow X
\end{equation*}
be the blowup of~$X$ at~$x$ or at~$\Theta(x)$, respectively.
We denote by~$H$ (the pullback to~$\tX$ of) the anticanonical class of~$X$ and 
by~$E$ the exceptional divisor of~$\sigma$.

First, note that for $m = m(x)$ the anticanonical linear system
\begin{equation*}
|-K_\tX| = |H - mE|
\end{equation*}
is base-point free by Theorem~\ref{th:bht} and~\cite[Lemma~5.7 and~5.5]{KP19}.
Moreover, combining~\cite[(5.1.9) and~(5.1.7)]{KP19}, we can uniformly write 
\begin{equation}
\label{eq:hhe}
H^3 = 2g - 2,
\qquad 
H^2 \cdot E = 0,
\qquad 
H \cdot E^2 = 2(m - 2),
\qquad 
E^3 = m - 1,
\end{equation} 
where we recall from Table~\ref{table:fanos} that~$g = \g(X) = 15$.
We will also need the following observation.

\begin{lemma}
\label{lemma:m}
The linear system $\MMM := |H - (m+1)E|$
on the blowup~$\tX$ of~$X$ has positive dimension
\begin{equation}
\label{eq:dim-m}
\dim \MMM \ge g - m - 7 \ge 6,
\end{equation} 
and has no fixed components.
\end{lemma}

\begin{proof}
The dimension is estimated in~\cite[Lemma~5.4(iii) and~(i)]{KP19}.
To prove that $\MMM$ has no fixed components, note that the linear system~$|kE|$ is 0-dimensional for any~$k \ge 0$
(since~$E$ is the exceptional divisor of a blowup),
hence the only possibility for a fixed component of~$\MMM$ is provided by the divisor~$E$ with some multiplicity.
So, assume 
\begin{equation*}
|H - (m+1)E|=(a - m - 1)E + |H - aE|, 
\end{equation*}
where $a\ge m + 2$ and $E$ is not a fixed component of the linear system~$|H - aE|$. 
Since the linear system~$|H - mE|$ is base-point free and $|H - aE|$ has no fixed components, using~\eqref{eq:hhe} we obtain
\begin{equation*}
0 \le (H - aE)^2\cdot (H - mE)= 2g-2-a^2(m^2 - 3m + 4) + 4am(m-2).
\end{equation*}
When $m = 2$ this gives $a^2 \le 14$, hence $a \le 3$, and when $m = 1$ this 
gives $(a+1)^2 \le 15$, hence~\mbox{$a \le 2$}.
In both cases this contradicts the assumption~$a \ge m + 2$.
\end{proof}

Now we can deduce the existence of the Sarkisov links.

\begin{proposition}
\label{prop:sl-x44-existence}
Let $X$ be a Fano threefold of type~\type{4,4} with a $\kk$-point~$x$.
\begin{enumerate}
\item 
If~$\rF_1(X,x) = \varnothing$, there exists a Sarkisov 
link~\eqref{eq:sl-general}, 
where~$\sigma$ is the blowup of~$x$.
\item 
If~$\rF_1(X,x) \ne \varnothing$, there exists a Sarkisov 
link~\eqref{eq:sl-general}, 
where $\sigma$ is the blowup of~$\Theta(x)$.
\end{enumerate}
In both cases the link is defined over~$\kk$.
\end{proposition}

\begin{proof}
We use notation~\eqref{eq:m-x}.
Recall that the anticanonical class $H = -K_X$ is very ample
and the image of the anticanonical embedding $X \subset \PP^{g + 1}$ is an 
intersection of quadrics (see Theorem~\ref{th:bht}).
The anticanonical morphism $\phi \colon \tX \to \PP^{g-m-1}$ cannot contract a divisor~$D$,
because by~\cite[Lemma~5.7 and~5.5]{KP19} this divisor is then a fixed component of~$\MMM$, 
but by Lemma~\ref{lemma:m} this linear system has no fixed components.
Therefore, the required link exists and is defined over~$\kk$ by~\cite[Lemma~5.7 and~5.5]{KP19}.
\end{proof} 

\subsection{The second contraction}
\label{subsec:sl-x44-details}

By Proposition~\ref{prop:sl-x44-existence} we have the 
diagram~\eqref{eq:sl-general},
so to finish the proof of Theorem~\ref{theorem:x44-links} it remains to describe the extremal contraction~$\sigma_+$.
During this step we systematically use the classification of extremal contractions from~\cite{Mori-1982} and~\cite{Cutkosky-1988}.

We denote by $H^+,E^+ \in \Pic(\tX^+)$ the strict transforms of the 
classes~$H,E \in \Pic(\tX)$.
Note that 
\begin{equation*}
- K_{\tX^+} = H^+ - mE^+
\end{equation*}
because $-K_\tX = H - mE$ by definition of~$\tX$ and the map~$\psi$ is an isomorphism in codimension one.
Consider also the strict transform 
\begin{equation*}
\MMM^+ := |H^+ - (m+1)E^+|
\end{equation*}
of the linear system~$\MMM$.

We denote by $\Upsilon_i \subset \tX$ the flopping curves and by~$\Upsilon_i^+ 
\subset \tX^+$ the corresponding flopped curves.
Finally, when $\rF_1(X,x) = \varnothing$ we denote by~$C$ a general twisted 
cubic curve on~$X$ passing through~$x$
and otherwise we denote by~$C$ a general conic meeting~$\Theta(x)$
(recall Lemma~\ref{lemma:cubics-x} and Lemma~\ref{lemma:conics-c0} for the 
description of the corresponding Hilbert schemes).
Note that
\begin{equation}
\label{eq:h-dot-c}
H \cdot C = m + 1.
\end{equation}
We denote by~$\tC$ the strict transform of~$C$ in~$\tX$
and by~$\tC^+$ the strict transform of~$\tC$ in~$\tX^+$.
Note that by Remark~\ref{rem:f2x-theta} and Remark~\ref{rem:f3x-x} the curve~$\tC$ is smooth and
\begin{equation}
\label{eq:e-dot-c}
E \cdot \tC = 1;
\end{equation}
in particular, $(H - mE)\cdot \tC = 1$ and~$\tC$ does not contain the curves~$\Upsilon_i$.

\begin{lemma}
\label{lemma:sigma+}
The nef cone of~$\tX^+$ is generated by the anticanonical class $-K_{\tX^+}$ and~$M^+ \in \MMM^+$,
and the Mori cone of~$\tX^+$ is generated by the class of the 
curves~$\Upsilon_i^+$ and the class of\/~$\tC^+$. 
In particular, the extremal contraction~$\sigma_+$ is given by a multiple of the linear 
system~$\MMM^+$ and contracts the extremal ray generated by~$\tC^+$.
\end{lemma}

\begin{proof}
Since~$\phi$ is crepant and~$\psi$ is a flop, the morphism~$\phi_+$ is crepant as well.
Moreover, the anticanonical linear system~$|-K_\tX|$ is base-point free 
by~\cite[Lemma~5.7]{KP19}, hence~$|-K_{\tX^+}|$ is base-point free as well.

On the other hand, we have~$(H - mE) \cdot \Upsilon_i = -K_{\tX} \cdot \Upsilon_i = 0$, 
and since $H \cdot \Upsilon_i > 0$, we conclude that $E \cdot \Upsilon_i > 0$.
Therefore, for~$M \in \MMM$ we have
\begin{equation*}
M \cdot \Upsilon_i = 
(H - (m+1)E)\cdot \Upsilon_i = 
-E \cdot \Upsilon_i < 0.
\end{equation*}
If $M^+ \in \MMM^+$ is the strict transform of~$M$, 
this implies that $M^+ \cdot \Upsilon_i^+ > 0$ by one of the definitions of a flop, see, e.g.,~\cite[Definition~6.10]{Kollar-Mori:88}.
Now if~$M^+$ is not nef, it is negative on the extremal ray~$\rR$ 
corresponding to the contraction~$\sigma_+ \colon \tX^+ \to X^+$.
Since the canonical class is also negative on~$\rR$, the contraction~$\sigma_+$ cannot be small 
(see \cite[Theorem~0]{Benveniste} or \cite[Corollary~6.3.4]{MP:1pt}),
hence curves in~$\rR$ sweep a subvariety of~$\tX^+$ of dimension~$\ge 2$,
hence the base locus of~$\MMM^+$ is at least 2-dimensional, which contradicts Lemma~\ref{lemma:m}.
This proves that~$M^+$ is nef.

Now we have $-K_\tX \cdot \tC = (H - mE) \cdot \tC = 1$ by~\eqref{eq:h-dot-c} and~\eqref{eq:e-dot-c}, 
hence $-K_{\tX^+} \cdot \tC^+ = 1$.
On the other hand, since a general divisor~$H$ meets~$C$ away from the 
indeterminacy locus of the map~$X\dashrightarrow \tilde X^+$, 
we have~$H^+\cdot \tC^+ \ge H\cdot C=m+1$ and so~$E^+\cdot \tC^+ \ge 1$.
Thus,
\begin{equation*}
M^+ \cdot \tC^+ = (-K_{\tX^+} - E^+) \cdot \tC^+ = 1 - E^+ \cdot \tC^+ \le 0.
\end{equation*}
Since~$M^{+}$ is nef, $M^+ \cdot \tC^+= 0$.

Combining the above computations we conclude that the nef cone of~$\tX^+$ is 
generated by~$-K_{\tX^+}$ and~$M^+$
and the Mori cone is generated by~$\Upsilon_i^+$ and~$\tC^+$.
The rest of the lemma follows from the Mori contraction 
theorem~\cite[Theorems~3.1 and~3.2]{Mori-1982}.
\end{proof}

Now we can finally prove Theorem~\ref{theorem:x44-links}.

\begin{proof}[Proof of Theorem~\xref{theorem:x44-links}]
Let $X$ be a Fano threefold of type~\type{4,4}.
By Proposition~\ref{prop:sl-x44-existence} 
there exists a Sarkisov link~\eqref{eq:sl-general} 
and it remains to describe the contraction~$\sigma_+$.
Since~$\sigma_+$ is an extremal contraction, we have $\uprho(X^+) = 
\uprho(\tX^+) - 1 = \uprho(\tX) - 1 = \uprho(X)$, hence~$\uprho(X^+) = 1$.
Similarly, we have~$\uprho(X^+_\bkk) \le \uprho(\tX^+_\bkk) - 1 = 
\uprho(\tX_\bkk) - 1 = \uprho(X_\bkk) = 2$, hence
\begin{equation}
\label{eq:rho-xplus-bkk}
\uprho(X^+_\bkk) \le 2.
\end{equation} 
On the other hand, in the case $\rF_1(X,x) \ne \varnothing$, the 
varieties~$\tX$ and~$\tX^+$ are not smooth and 
arguing as in the proof of~\cite[Theorem~5.9]{KP19} we obtain 
\begin{equation}
\label{eq:rk-cl-txp}
\rk \Cl(\tX^+) = 2,
\qquad 
\rk \Cl(\tX^+_\bkk) = 5 - m.
\end{equation}

Since~$\phi$ and~$\phi_+$ are crepant morphisms, the projection formula implies 
that any triple intersection product of divisor classes on~$\tX^+$ which 
includes~$K_{\tX^+}$
is equal to the analogous triple product on~$\tX$, so 
using~\eqref{eq:hhe}
we compute
(recall that~$g = \g(X) = 15$)
\begin{equation}
\label{eq:new1}
(-K_{\tilde X^+})^3 = 2(g - m - 3) {} = 24 - 2m,\quad 
(-K_{\tilde X^+})^2\cdot E^+ = 4,\quad 
(-K_{\tilde X^+})\cdot (E^+)^2 = -2.
\end{equation}
On the other hand, by Lemma~\ref{lemma:sigma+} and primitivity of~$H^+ - 
(m + 1)E^+$ we have 
\begin{equation}
\label{eq:new2}
H^+ - (m + 1)E^+ = \sigma_+^*A^+,
\end{equation}
where $A^+$ is the ample generator of the Picard group of~$X^+$.
We have
\begin{equation}
\label{eq:intersections}
(\sigma_+^*A^+)^2 \cdot(-K_{\tX^+}) = (H^+ - (m + 1)E^+)^2\cdot (-K_{X^+}) = 
2(g - m - 8) = 14 - 2m > 0,
\end{equation}
therefore~$\sigma_+$ is not a del Pezzo fibration.
Similarly, if~$\sigma_+$ is a conic bundle, it follows that
\begin{equation*}
(A^+)^2 = 7 - m, 
\end{equation*}
hence~$X^+$ is a smooth quintic or sextic del Pezzo surface,
which of course contradicts the inequality~\eqref{eq:rho-xplus-bkk}.
Therefore, the morphism~$\sigma_+$ is birational.

By Lemma~\ref{lemma:sigma+} the morphism~$\sigma_+$ contracts the strict 
transform~$R^+$ of the divisor swept by curves~$C$,
i.e., the strict transform of the divisor~$R_x \subset X$ if~$\rF_1(X,x) = 
\varnothing$, or of the divisor~$R_{\Theta(x)}$ otherwise.
In both cases Lemma~\ref{lemma:cubics-x} and Lemma~\ref{lemma:conics-c0} show 
that~$R^+$ 
has over~$\bkk$ two irreducible components swapped by the Galois group.
Therefore, it follows from~\eqref{eq:rk-cl-txp} that 
\begin{equation*}
\rk \Cl(X^+) = 1,
\qquad
\rk \Cl(X^+_\bkk) = 3 - m,
\end{equation*}
and~$\sigma_+$ is the blowup of two $\bkk$-curves or two $\bkk$-points.
Furthermore, by Lemmas~\ref{lemma:cubics-x} and~\ref{lemma:conics-c0} we have
\begin{equation*}
R^+ \sim H^+ - (m + 2)E^+.
\end{equation*}
Denoting by~$i_+$ the index of~$X^+$ and by~$a_+$ the discrepancy of the exceptional divisor~$R^+$ of~$\sigma_+$, 
and computing the anticanonical class of~$\tX^+$ in two ways we obtain the equality
\begin{equation*}
H^+ - mE^+ = i_+(H^+ - (m+1)E^+) - a_+(H^+ - (m+2)E^+).
\end{equation*}
Solving this equation, we obtain~$i_+ = 2$ and~$a_+ = 1$.
Thus, $X^+$ is a Fano threefold of index~$2$ 
and~$\sigma_+$ is either the blowup of a $\kk$-irreducible curve~$B^+$ with two~$\bkk$-components, 
or of two rational double $\bkk$-points on~$X^+$ swapped by the Galois action.
Moreover, using the equality from~\eqref{eq:intersections} we obtain
\begin{equation*}
14 - 2m= (\sigma_+^*A^+)^2\cdot (-K_{\tX^+}) = (\sigma_+^*A^+)^2 \cdot 
(2\sigma_+^*A^+ - R^+) = 2(\sigma_+^*A^+)^3,
\end{equation*}
hence $X^+$ is a quintic or sextic del Pezzo threefold, respectively.
Finally, if~$X^+$ is singular, its class group~$\Cl(X^+)$ has rank greater than~$1$ (see~\cite[Theorem~1.7]{Prokhorov-GFano-1}),
which contradicts to the equality~$\rk\Cl(X^+) = 1$ obtained above.
Thus, $X^+$ is smooth and~$\sigma_+$ is the blowup of a curve~$B^+$ 
such that~$B^+_{\bkk}$ has two irreducible~$\bkk$-components swapped by the Galois group.

To finally compute the degree of~$B^+$, recall that~$R^+$ is the exceptional divisor of~$\sigma_+$. 
Note that on the one hand,  equalities~\eqref{eq:new1} and~\eqref{eq:new2} imply that
\begin{equation*}
(\sigma_+^*A^+) \cdot (-K_{\tX^+})^2 = 2g - 2m - 10 = 20 - 2m,
\end{equation*}
and on the other hand, this expression is equal to
\begin{equation*}
(\sigma_+^*A^+) \cdot (2\sigma_+^*A^+ - R^+)^2 = 
4(\sigma_+^*A^+)^3 + (\sigma_+^*A^+) \cdot (R^+)^2 = 4(7 - m)- \deg(B^+)
\end{equation*}
(where the degree is computed with respect to~$A^+$).
Thus, 
\begin{equation*}
\deg(B^+) = 8 - 2m,
\end{equation*}
hence~$B^+$ is a quartic or sextic curve with two connected $\bkk$-components (swapped by the Galois action), 
i.e., a union of two conics or two cubic curves.
\end{proof}

\subsection{Rationality}

Now we use the constructed links to prove rationality of threefolds of 
type~\type{4,4}.

\begin{proposition}
\label{prop:x44}
Let $X$ be a Fano threefold of type~\type{4,4}.
If $X(\kk) \ne \varnothing$ then~$X$ is $\kk$-rational. 
\end{proposition}

\begin{proof}
Let~$x \in X(\kk)$ be a~$\kk$-point.
First, assume that~$\rF_1(X,x) = \varnothing$.
Then by Theorem~\ref{theorem:x44-links}\ref{prop:sl:point-2-blowup-Q} 
the variety~$X$ is birational to a smooth quintic del Pezzo threefold~$X^+$.
But~$X^+$ is $\kk$-rational by~\cite[Theorem~3.3]{KP19}, hence so is~$X$.

Now assume that $\rF_1(X,x) \ne \varnothing$. 
Then by Theorem~\ref{theorem:x44-links}\ref{prop:sl-conic} 
the variety $X$ is birational to a smooth Fano threefold $X^+$ of 
type~\type{2,2}.
Moreover, by the Nishimura lemma we have~$X^+(\kk) \ne \varnothing$.
Therefore, $X^+$ is $\kk$-rational by~Proposition~\ref{prop:x22}, hence so 
is~$X$.
\end{proof} 

\section{Fano threefolds of type~\type{3,3}}

In this section we prove that a Fano threefold~$X$ of type~\type{3,3} is 
$\kk$-unirational if~$X(\kk) \ne \varnothing$, 
but not $\kk$-rational if~$\uprho(X) = 1$.

\subsection{The discriminant curve}
\label{subsec:discriminant}

Let~$X$ be a Fano threefold of type~\type{3,3} with~$X(\kk) \ne \varnothing$.
Recall from Lemma~\ref{lemma:picard} that the image~$\rG_X$ of the Galois group~$\Gal(\bkk/\kk)$ in~$\Aut(\Pic(X_\bkk))$ 
is a group of order~$2$ swapping the generators~$H_1$ and~$H_2$ of~$\Pic(X_\bkk)$. 
The homomorphism $\Gal(\bkk/\kk)\to \rG_X$ therefore defines a quadratic extension~$\kk'/\kk$ 
such that $H_1$ and~$H_2$ are defined on~$X_{\kk'}$, hence
\begin{equation*}
X_{\kk'} \cong \Big(\PP(V_1) \times \PP(V_2)\Big) \cap \PP(A^\perp),
\end{equation*}
where~$V_i$ are $\kk'$-vector spaces of dimension~$4$ and~$A \subset V_1^\vee \otimes V_2^\vee$ 
is the $3$-dimensional subspace of linear equations of~$X_{\kk'}$.
Note that the~$\kk'$-spaces $V_1 \otimes V_2$ and~$A$ are defined over~$\kk$, 
as well as the inclusion~$A \subset V_1^\vee \otimes V_2^\vee$.
We think of vectors~$a \in A$ as of bilinear forms on~$V_1 \otimes V_2$ and denote by 
\begin{equation*}
\Gamma \xhookrightarrow{\ \ } \PP(A)
\end{equation*}
the \textsf{discriminant curve} parameterizing degenerate bilinear forms; it is also 
defined over~$\kk$.

\begin{lemma}
\label{lemma:gamma-x33}
The curve~$\Gamma$ is a smooth plane quartic curve; 
in particular it is a non-hyperelliptic curve of genus~$3$.
\end{lemma}

\begin{proof}
The discriminant divisor in~$\PP(V_1^\vee \otimes V_2^\vee)$, i.e., the divisor parameterizing degenerate bilinear forms,
is a quartic hypersurface, hence~$\Gamma$ is a quartic curve or the entire plane.
To prove that~$\Gamma$ is a smooth curve we can work over~$\kk'$,
and it is enough to show that the tangent space to~$\Gamma$ at any point is 1-dimensional.
Assume to the contrary that the tangent space at a point~$[a] \in \PP(A)$ is 2-dimensional; then
\begin{enumerate}
\item 
\label{item:rank-a}
either the bilinear form~$a(-,-) \in V_1^\vee \otimes V_2^\vee$ has corank at least~$2$, 
\item 
\label{item:tangency-a}
or~$a$ has corank~$1$ and if the vectors~$v_1 \in V_1$ and~$v_2 \in V_2$ generate its left and right kernels, respectively, 
the point $([v_1],[v_2]) \in \PP(V_1) \times \PP(V_2)$ belongs to~$X_{\kk'}$.
\end{enumerate}
In case~\ref{item:rank-a}, if~$K_1 \subset V_1$ and~$K_2 \subset V_2$ are the left and right kernels of~$a$ 
(they have dimension~$\ge 2$ by assumption), the form~$a$ vanishes on~$\PP(K_1) \times \PP(K_2)$, 
hence the intersection~$\big(\PP(K_1) \times \PP(K_2)\big) \cap X_{\kk'}$ 
is a codimension-2 linear section of~$\PP(K_1) \times \PP(K_2)$, hence it is non-empty.
Therefore, in case~\ref{item:rank-a}, similarly to the
case~\ref{item:tangency-a}, 
there is a point $([v_1],[v_2]) \in X_{\kk'}$ such that~$v_1$ and~$v_2$ belong to the left and right kernels of some~$a$.
Then the hyperplane section of~$\PP(V_1) \times \PP(V_2)$ by the hyperplane corresponding to~$a$ is singular at~$([v_1],[v_2])$,
hence~$X_{\kk'}$ is also singular at this point.
\end{proof}

As explained in~Lemma~\ref{lemma:pi-i}\ref{lemma:pi-i:33}
the projections~$\pi_1 \colon X_{\kk'} \to \PP(V_1)$ and~$\pi_2 \colon X_{\kk'} \to \PP(V_2)$ defined over~$\kk'$, but not over~$\kk$,
are the blowups of curves $\Gamma_i\subset \PP(V_i)$ of genus~$3$ and degree~$6$ also defined over~$\kk'$.
The next lemma relates the $\kk'$-curves~$\Gamma_i$ to the discriminant curve~$\Gamma$ defined over~$\kk$.

\begin{lemma}
\label{lemma:gamma-gamma-i}
There is a natural isomorphism $\Gamma_i \cong \Gamma_{\kk'}$ of curves 
over~$\kk'$.
\end{lemma}

\begin{proof}
The fiber of the projection~$\pi_1$ over a point $[v_1] \in \PP(V_1)$ 
is the intersection of the projectivizations of the orthogonals of~$v_1$ with respect to all bilinear forms~$a \in A$.
Therefore, it has positive dimension if and only if~$v_1$ belongs to the left kernel of one of the forms.
Furthermore, if~$v_1$ belongs to the left kernel of two distinct forms in~$A$, 
the fiber of~$\pi_1$ over~$[v_1]$ contains a plane, which contradicts Corollary~\ref{cor:planes}.
This means that the morphism
\begin{equation*}
\gamma_i \colon \Gamma_{\kk'} \longrightarrow \PP(V_i), 
\qquad 
a \longmapsto \Ker_i(a),
\end{equation*}
where $\Ker_1$ and $\Ker_2$ denote the left and right kernels of the bilinear form~$a$, respectively,
is an isomorphism $\Gamma_{\kk'} \to \Gamma_i$.
\end{proof}

\begin{remark}
It is also easy to check that if $H_i\vert_\Gamma$ 
are the pullbacks of the hyperplane classes of~$\Gamma_i\subset \PP(V_i)$ to~$\Gamma_{\kk'}$ 
under the isomorphism of Lemma~\ref{lemma:gamma-gamma-i} then~$H_1\vert_\Gamma + H_2\vert_\Gamma = 3K_\Gamma$
and that the divisor classes~$H_i\vert_\Gamma - K_\Gamma$ are non-effective and swapped by the~$\Gal(\kk'/\kk)$-action.
Conversely, given two such classes on a curve~$\Gamma_{\kk'}$ one can reconstruct the variety~$X$.
\end{remark}

\subsection{The double projection from a point}
\label{subsec:x33-double-projection}

Recall the quadratic extension~$\kk'/\kk$ defined in~\S\ref{subsec:discriminant}.
Recall also the canonical embedding~$X \subset Y$, where~$Y$ is a $\kk$-form of~$\PP^3 \times \PP^3$.
We consider the birational transformation of Theorem~\ref{proposition:toric-link} 
for the variety~$Y_{\kk'} = \PP(V_1) \times \PP(V_2)$ associated with a~$\kk$-point
\begin{equation*}
x_0 = ([v_1],[v_2]) \in X \subset Y.
\end{equation*}
As in~\S\ref{subsec:ppp} we denote $\barV_i := V_i / \kk' v_i$ and choose a splitting~$V_i = \kk' v_i \oplus \bar{V}_i$.
The transformation of Theorem~\ref{proposition:toric-link} in this case looks as follows
\begin{equation}
\label{eq:link-y33}
\vcenter{\xymatrix@C=4em{
\tY \ar[d]_\sigma 
\ar@{-->}[rr]^-\psi &&
\tY^+ \ar[d]^{\sigma_+} 
\\
Y &
&
Y^+
}}
\end{equation}
where~$\sigma$ is the blowup of~$x_0$, $Y^+_{\kk'} \cong \PP(\bar{V}_1) \times 
\PP(\bar{V}_2)$, 
$\sigma_+$ is the projectivization of the vector bundle
\begin{equation}
\label{eq:ce-x33}
\cE = \cO(-h_1-h_2) \oplus \cO(-h_1) \oplus \cO(-h_2)
\end{equation}
(here $h_i$ stand for the hyperplane classes of~$\PP(\barV_i)$) over~$Y^+_{\kk'}$,
and the map~$\psi$ is a small birational isomorphism.
Note that all varieties and maps in~\eqref{eq:link-y33} are defined over~$\kk$. 

Recall also the relations~\eqref{eq:toric-pic-1} in~$\Pic(\tY_{\kk'}) = 
\Pic(\tY^+_{\kk'})$ 
between the hyperplane classes $H_i$ of the factors~$\PP(V_i)$ of~$Y_{\kk'}$, 
the class~$E$ of the exceptional divisor of~$\sigma$,
the hyperplane classes~$h_i$, and the relative hyperplane class~$h$ 
of~$\tY^+_{\kk'} = \PP_{Y^+_{\kk'}}(\cE)$:
\begin{equation}
\label{eq:pic-relations-x33}
\begin{cases}
h_1 = H_1 - E,\\
h_2 = H_2 - E,\\
h_{\phantom{2}} = H_1 + H_2 - E
\end{cases}
\qquad 
\begin{cases}
\hbox to 1.2em{$H_1$} = h - h_2,\\
\hbox to 1.2em{$H_2$} = h - h_1,\\
\hbox to 1.2em{$E$} = h - h_1 - h_2.
\end{cases}
\end{equation}

Since~$X$ is a smooth linear section of~$Y$, containing the point~$x_0$, 
it is a complete intersection of three divisors $D_\alpha$, $1 \le \alpha \le 3$, in the linear system~$|H_1 + H_2|$
whose strict transforms on~$\tY$ belong to the linear system~$|H_1 + H_2 - E|$.
Now it follows from~\eqref{eq:pic-relations-x33} that their strict 
transforms~$\tD^+_\alpha$ on~$\tY^+$ belong to the linear system~$|h|$.
As in Proposition~\ref{prop:x-transform} we consider the complete intersection
\begin{equation*}
\tX^+_{\kk'} := \tD^+_1 \cap \tD^+_2 \cap \tD^+_3 \subset \PP_{Y^+_{\kk'}}(\cE).
\end{equation*}
It follows that~$\tX^+_{\kk'}$ is determined by a morphism of vector bundles
\begin{equation*}
\xi \colon \cE \longrightarrow A^\vee \otimes \cO,
\end{equation*}
and if we choose a basis~$a_1,a_2,a_3$ in~$A$, it is easy to see that~$\xi$ is 
given by the matrix
\begin{equation}
\label{eq:xi-x33}
\xi = 
\begin{pmatrix}
\bar{a}_1(-,-) & a_1(-,v_2) & a_1(v_1,-) \\
\bar{a}_2(-,-) & a_2(-,v_2) & a_2(v_1,-) \\
\bar{a}_3(-,-) & a_3(-,v_2) & a_3(v_1,-) 
\end{pmatrix},
\end{equation} 
where~$\bar{a}_i \in \bar{V}_1^\vee \otimes \bar{V}_2^\vee$ denotes the restriction of the bilinear form~$a_i$ to~$\bar{V}_1 \otimes \bar{V}_2$,
while~$a_i(-,v_2)\in V_1^\vee$ and~$a_i(v_1,-)\in V_2^\vee$ are considered as linear functions 
on~$\bar{V}_1$ and~$\bar{V}_2$, respectively.

\begin{proposition}
\label{prop:x33-plus}
The threefold~$X$ is $\kk$-birational to the $\kk$-form~$\tX^+$ of the threefold~$\tX^+_{\kk'}$ defined by~\eqref{eq:xi-x33}
and to a $\kk$-form~$X^+$ of its image in~$Y^+$
\begin{equation*}
X^+_{\kk'} = \sigma_+(\tX^+_{\kk'}) = \{ \det(\xi) = 0 \} \subset \PP(\bar{V}_1) \times \PP(\bar{V}_2),
\end{equation*}
which is a geometrically irreducible and normal divisor of bidegree~$(2,2)$.
Moreover, 
\begin{itemize}
\item 
if~$\rF_1(X,x_0) = \varnothing$ then $\tX^+ \cong \tX = \Bl_x(X)$ is smooth, 
the morphism~$\sigma_+ \colon \tX^+ \to X^+$ is induced by the double projection from~$x_0$,
and it is a small resolution of singularities;
\item 
if~$\rF_1(X,x_0) \ne \varnothing$ then~$X^+$ contains a $\kk$-form 
of a quadric surface~$\PP^1 \times \PP^1 \subset \PP(\bar{V}_1) \times \PP(\bar{V}_2)$ rational over~$\kk$.
\end{itemize}
\end{proposition}
\begin{proof}
To prove birationality of~$\tX$ and~$\tX^+ = \psi_*(\tX)$ we 
apply Proposition~\ref{prop:x-transform},
so we need to verify the dimension conditions~\eqref{eq:dimension-conditions}.
We have~$(Y_1)_{\kk'} = [v_1] \times \PP(V_2)$, hence
\begin{equation*}
X_{\kk'} \cap (Y_1)_{\kk'} = ([v_1] \times \PP(V_2)) \cap \PP(A^\perp)
\end{equation*}
is a fiber of the projection~$\pi_1 \colon X_{\kk'} \to \PP(V_1)$.
By Lemma~\ref{lemma:pi-i} it is a point or a line.
A similar argument for~$X_{\kk'} \cap (Y_2)_{\kk'}$ shows that the first part of~\eqref{eq:dimension-conditions} holds.
Moreover, this argument also shows that in the case~$\rF_1(X,x_0) = \varnothing$, the blowup~$\tX_{\kk'}$ of~$X_{\kk'}$
has empty intersection with the indeterminacy locus~$(\tY_1)_{\kk'} \sqcup (\tY_2)_{\kk'}$ of the map~$\psi$.

On the other hand, the subbundle~$\bar{\cE}_1 \subset \cE$ is just the 
summand~$\cO(-h_2)$ in~\eqref{eq:ce-x33}, 
hence the corresponding intersection~$\tX^+_{\kk'} \cap 
\PP_{Y^+_{\kk'}}(\bar{\cE}_1)$ is the zero locus of the morphism
\begin{equation*}
\xi_2 \colon \cO(-h_2) \longmapsto A^\vee \otimes \cO
\end{equation*}
given by the last column of~\eqref{eq:xi-x33}.
It is easy to see that this is empty, if~$\rF_1(X,x_0) = \varnothing$, or isomorphic to a line otherwise.
A similar argument works for~$\tX^+_{\kk'} \cap \PP_{Y^+_{\kk'}}(\bar{\cE}_2)$;
therefore, the second part of~\eqref{eq:dimension-conditions} also holds.
This proves that~$\tX^+ = \psi_*(\tX)$ is a $\kk$-form of~$\tX^+_{\kk'}$, which is $\kk$-birational to~$X$, 
and if~$\rF_1(X,x_0) = \varnothing$, it is isomorphic to~$\tX$, 
and in particular in this case it is smooth.

Now we describe the image of~$\tX^+$ in~$Y^+$.
By definition, $\tX^+_{\kk'}$ parameterizes points in the projectivizations of kernel spaces of~$\xi$;
therefore, its image in~$Y^+_{\kk'} = \PP(\bar{V}_1) \times \PP(\bar{V}_2)$ 
is the degeneracy locus~$X^+_{\kk'}$ of~$\xi$
which is, of course, given by the equation~\mbox{$\det(\xi) = 0$}.
Since~$\det(\cE) \cong \cO(-2h_1 - 2h_2)$ by~\eqref{eq:ce-x33}, this is a divisor of bidegree~$(2,2)$, which is 
geometrically irreducible because~$\tX^+_{\kk'}$ is.
Moreover, fibers of the morphism~\mbox{$\sigma_+ \colon \tX^+_{\kk'} \to X^+_{\kk'}$} are linear spaces, 
so since both the source and the target are 3-dimensional, the morphism is birational.
To prove that~$X^+_\bkk$ is normal we consider the Koszul resolution
\begin{equation*}
0 \longrightarrow 
\cO_{\PP_{Y^+_\bkk}(-\cE)}(-3h) \longrightarrow 
\cO_{\PP_{Y^+_\bkk}(-\cE)}(-2h)^{\oplus 3} \longrightarrow 
\cO_{\PP_{Y^+_\bkk}(-\cE)}(-h)^{\oplus 3} \longrightarrow
\cO_{\PP_{Y^+_\bkk}(-\cE)} \longrightarrow 
\cO_{\tX^+_\bkk} \longrightarrow 0.
\end{equation*}
Pushing it forward to~$Y^+_\bkk$, we obtain the following exact sequence
\begin{equation*}
0 \longrightarrow \cO_{Y^+_\bkk}(-2h_1-2h_2) \longrightarrow \cO_{Y^+_\bkk} \longrightarrow \sigma_{+*}\cO_{\tX^+_\bkk} \longrightarrow 0.
\end{equation*}
It follows that~$\sigma_{+*}\cO_{\tX^+_\bkk} \cong \cO_{X^+_\bkk}$, and since~$\tX^+_\bkk$ is normal, so is~$X^+_\bkk$.

Now assume~$\rF_1(X,x_0) = \varnothing$.
In this case the pullback along~$\sigma_+$ of the ample class~$h_1 + h_2$ 
on~$\PP(\bar{V}_1) \times \PP(\bar{V}_2)$
by~\eqref{eq:pic-relations-x33} equals~$H_1 + H_2 - 2E$, the anticanonical 
class of~$\tX^+ \cong \tX$, 
hence the morphism~$\sigma_+$ is the double projection from the point~$x_0$.
Consequently, it is small by the argument of~\cite[Theorem~5.17]{KP19}.
Indeed, by~\cite[Lemma~5.4(iii)]{KP19} we have~$\dim|H_1 + H_2 - 3E| \ge g - 9 = 2$
(recall that~$g = \g(X) = 11$, see Table~\ref{table:fanos}),
hence by~\cite[Lemma~5.7(ii)]{KP19} any divisor~$D$ contracted by~$\sigma_+$ 
must be a fixed component of~$|H_1 + H_2 - 3E|$,
and at the same time by~\cite[(5.1.8)]{KP19} its class should be a multiple 
of~$H_1 + H_2 - 5E$, 
and these two conclusions are incompatible. 

Finally, assume that~$\rF_1(X,x_0) \ne \varnothing$.
As it was explained in Lemma~\ref{lemma:gamma-gamma-i} this means that (for appropriate $\kk'$-basis in~$A$) we have 
\begin{equation*}
a_1(v_1,-) = 0
\qquad\text{and}\qquad
a_2(-,v_2) = 0
\end{equation*}
as linear functions on~$\bar{V_2}$ and~$\bar{V}_1$, respectively; moreover, 
$[a_1],[a_2] \in \PP(A)$ as above are unique and swapped by the Galois action.
Consider the surface
\begin{equation*}
\{a_1(-,v_2) = 0,\ a_2(v_1,-) = 0\} \subset \PP(\bar{V}_1) \times \PP(\bar{V}_2) 
= Y^+_{\kk'}.
\end{equation*}
(isomorphic to~$\PP^1_{\kk'} \times \PP^1_{\kk'}$).
The equations, defining it are Galois-conjugate, hence it comes from a~$\kk$-surface in~$Y^+$.
This surface is the image of the exceptional divisor~$E$ of~$\sigma$, hence it is~$\kk$-rational.
It is clear from~\eqref{eq:xi-x33} that this surface is contained in the degeneracy locus~$X^+$ of~$\xi$.
\end{proof} 

\subsection{A conic bundle structure}
\label{subsec:x33-conic-bundle}

In this section we work under the assumption~$\rF_1(X,x_0) = \varnothing$
and show that in this case~$X$ admits a nice conic bundle structure.

We will need a general result about what we call \textsf{Springer resolutions}.
Let~$M$ be a variety, let~$\xi \colon \cE_1 \to \cE_2^\vee$ be a morphism of vector bundles on~$M$ of the same rank 
and let~$\xi^\vee \colon \cE_2 \to \cE_1^\vee$ be its dual morphism.
Assume the degeneracy locus $Z \subset M$ of~$\xi$ is a geometrically integral 
divisor.
Let~$Z_1 \subset \PP_M(\cE_1)$ and~$Z_2 \subset \PP_M(\cE_2)$ be the zero loci 
of the morphisms
\begin{equation*}
\cO(-h_{\cE_1}) \hookrightarrow p_1^*\cE_1 \xrightarrow{\quad p_1^*\xi \quad} p_1^*\cE_2^\vee
\qquad\text{and}\qquad
\cO(-h_{\cE_2}) \hookrightarrow p_2^*\cE_2 \xrightarrow{\quad p_2^*\xi^\vee \quad} p_2^*\cE_1^\vee,
\end{equation*}
where $p_i \colon \PP_M(\cE_i) \to M$ are the projections,
$h_{\cE_i}$ are their relative hyperplane classes, and the first arrows are the tautological embeddings.

\begin{lemma}
\label{lemma:springer}
If one of the morphisms
\begin{equation*}
p_1\vert_{Z_1} \colon Z_1 \longrightarrow Z
\qquad\text{or}\qquad
p_2\vert_{Z_2} \colon Z_2 \longrightarrow Z
\end{equation*}
is birational then so is the other.
Moreover, if one of them is small then so is the other, and there is an equality
\begin{equation}
\label{eq:springer-pic}
(h_{\cE_1} + \rc_1(p_1^*\cE_1)) + (h_{\cE_2} + \rc_1(p_2^*\cE_2)) = 0
\end{equation}
in the group~$\Cl(Z_1) \cong \Cl(Z) \cong \Cl(Z_2)$, 
where the isomorphisms of the class groups are induced by the small 
birational morphisms~$p_i$.
\end{lemma}

\begin{proof}
Let~$Z^{\ge c} \subset Z$ be the locus of points where the corank of~$\xi$ is at 
least~$c$ (so that~$Z = Z^{\ge 1}$).
Then both morphisms~$p_i\vert_{Z_i}$ are~$\PP^{c-1}$-fibrations 
over~$Z^{\ge c} \setminus Z^{\ge c + 1}$.
In particular if one of the morphisms is birational then $\dim Z^{\ge c} \le \dim Z - c$ for~$c \ge 2$
and then the other morphism is also birational.
Similarly, if one of the morphisms is small then $\dim Z^{\ge c} \le \dim Z - c - 1$ for~$c \ge 2$
and then the other morphism is also small.
Finally, assuming that the morphisms are small, we have
\begin{equation*}
\Cl(Z_1) = \Cl(Z_1 \setminus p_1^{-1}(Z^{\ge 2})) = \Cl(Z \setminus Z^{\ge 2}) = 
\Cl(Z_2 \setminus p_2^{-1}(Z^{\ge 2})) = \Cl(Z_2),
\end{equation*}
and when restricted to~$Z \setminus Z^{\ge 2}$ the morphism~$\xi$ has constant corank~$1$, 
the summands in~\eqref{eq:springer-pic} are equal to~$\rc_1(\Ima(\xi))$ and~$\rc_1(\Ima(\xi^\vee))$, respectively,
and~\eqref{eq:springer-pic} follows from the natural duality isomorphism~$\Ima(\xi^\vee) \cong \Ima(\xi)^{\vee}$.
\end{proof}

Now, coming back to the threefold~$X$ of type~\type{3,3} and assuming~$\rF_1(X,x_0) = \varnothing$
we recall
that~$X^+ \subset Y^+$ is the degeneracy locus of~$\xi \colon \cE \to A^\vee \otimes \cO$
and note that~$\tX^+ \subset \PP_{Y^+}(\cE)$ is one of its Springer resolutions.
Consider the other Springer resolution
\begin{equation*}
\tX^{\ppl} \subset \PP_{Y^+}(A \otimes \cO) \cong Y^+ \times \PP(A),
\end{equation*}
which by definition is the zero locus of the morphism
\begin{equation*}
\cO(-h_A) \hookrightarrow A \otimes \cO \xrightarrow{\quad \xi^\vee \quad} 
\cE^\vee,
\end{equation*}
where $h_A$ is the hyperplane class of~$\PP(A)$ and we suppress the pullbacks in the notation.
In view of~\eqref{eq:ce-x33} the scheme~$X_{\kk'}^{\ppl}$ is just a complete intersection of divisors 
of types $h_1 + h_2 + h_A$, $h_1 + h_A$, and~$h_2 + h_A$ in~$\PP(\bar{V}_1) \times \PP(\bar{V}_2) \times \PP(A)$
that correspond to the columns of~\eqref{eq:xi-x33}.

\begin{proposition}
\label{prop:conic-bundle-x33}
If~$\rF_1(X,x_0) = \varnothing$ there is a commutative diagram defined 
over~$\kk$
\begin{equation}
\label{eq:link-x33}
\vcenter{\xymatrix@C=4em{
\tX \ar[d]_\sigma 
\ar@{=}[rr]^-\psi &&
\tX^+ \ar[dr]^{\sigma_+} 
\ar@{-->}[rr]^-{\psi_+} &&
\tX^{\ppl} \ar[dl]_{\sigma_{\ppl}} \ar[d]^f
\\
X &
&&
X^+ &
\PP(A),
}}
\end{equation}
where~$\tX^+$ and~$\tX^{\ppl}$ are the Springer resolutions of the degeneracy 
locus~$X^+ \subset Y^+$ of~$\xi$,
the morphisms~$\sigma_+$ and~$\sigma_{\ppl}$ are small birational contractions, 
and $\psi_+ = \sigma_{\ppl}^{-1} \circ \sigma_+$ is a flop.

Moreover, $\tX^{\ppl}$ is smooth, $f$ is a flat conic bundle whose discriminant curve is the curve~$\Gamma$ defined in~\textup\S\xref{subsec:discriminant},
the map $f \circ \psi_+ \circ \psi \colon \tX \dashrightarrow \PP(A)$ is given by the linear system $|H_1 + H_2 - 3E|$,
and the exceptional divisor~$E \subset \tX$ of~$\sigma$ dominates~$\PP(A)$.
\end{proposition}
\begin{proof}
The morphism~$\sigma_+$ is small by Proposition~\ref{prop:x33-plus}, 
hence~$\sigma_{\ppl}$ is small by Lemma~\ref{lemma:springer};
moreover, it follows that both morphisms are crepant.
Now, the relation~\eqref{eq:springer-pic} implies that the $\sigma_+$-antiample class~$-h$ 
is $\sigma_{\ppl}$-ample, hence~$\psi_+ := \sigma_{\ppl}^{-1} \circ \sigma_+$ is a flop.
Since~$\tX^+ \cong \tX$ is smooth and~$\psi_+$ is a flop, $\tX^{\ppl}$ is smooth as well, see~\cite[Theorem~2.4]{Kollar1989}. 

Next, we show that~$f$ is a conic bundle and identify its discriminant.
For this note that by definition the fiber of~$f$ over a point~$[a] \in \PP(A)$ is 
given in~$\PP(\bar{V}_1) \times \PP(\bar{V}_2)$ by the equations
\begin{equation*}
a(-,v_2) = a(v_1,-) = \bar{a}(-,-) = 0.
\end{equation*}
The first is a linear function on~$\bar{V}_1$, the second is a linear function 
on~$\bar{V}_2$, 
and both are non-zero because~$\rF_1(X,x_0) = \varnothing$, so their common zero 
locus is~$\PP^1_{\kk'} \times \PP^1_{\kk'} \subset \PP(\bar{V}_1) \times \PP(\bar{V}_2)$.
The last equation~$\bar{a}(-,-) = 0$ cuts a divisor of bidegree~$(1,1)$ on 
this~$\PP^1_{\kk'} \times \PP^1_{\kk'}$, i.e., a conic;
and if~$\bar{a}(-,-)$ vanishes identically, then the corresponding bilinear 
form~$a$ vanishes on~$\PP^2_{\kk'} \times \PP^2_{\kk'} \subset \PP(V_1) \times \PP(V_2)$,
hence has corank~$2$, which is impossible by the argument of 
Lemma~\ref{lemma:gamma-x33} as~$X$ is smooth.
This shows that~$f$ is a flat conic bundle.
Finally, note that if~$\bar{v}'_i$, $\bar{v}''_i$, $\bar{v}'''_i$ are bases of vector spaces~$\bar{V}_i$ 
such that~$a(v_1,\bar{v}'_2) = a(v_1,\bar{v}''_2) = a(\bar{v}'_1,v_2) = a(\bar{v}''_1,v_2) = 0$ then the matrix of~$a$ has the form
\begin{equation*}
\begin{pmatrix}
0 & 0 & 0 & a(v_1,\bar{v}'''_2) \\
0 & \bar{a}(\bar{v}'_1,\bar{v}'_2) & \bar{a}(\bar{v}'_1,\bar{v}''_2) & * \\
0 & \bar{a}(\bar{v}''_1,\bar{v}'_2) & \bar{a}(\bar{v}''_1,\bar{v}''_2) & * \\
a(\bar{v}'''_1,v_2) & * & * & *
\end{pmatrix}
\end{equation*}
with non-zero entries~$a(v_1,\bar{v}'''_2)$ and~$a(\bar{v}'''_1,v_2)$, and with the 2-by-2 matrix in the middle 
giving the equation of the conic~$f^{-1}(a)$ in~$\PP^1_\bkk \times \PP^1_\bkk$.
Therefore, the conic is singular if and only if~$\det(a) = 0$, i.e., if and only if~$[a] \in \Gamma$.
Thus, the discriminant curve of~$f$ equals~$\Gamma$.

Finally, using~\eqref{eq:springer-pic}, \eqref{eq:ce-x33}, and~\eqref{eq:pic-relations-x33} 
we deduce that the map~$f \circ \psi_+ \circ \psi$ is given by the linear system
\begin{equation*}
h_A = 2(h_1 + h_2) - h = H_1 + H_2 - 3E.
\end{equation*}
Since the canonical class of~$\tX$ is equal to~$H_1 + H_2 - 2E$ and~$\psi_+$ is a flop, 
it follows that~$\psi_*(E)$ is a relative anticanonical divisor for~$f$, hence it dominates~$\PP(A)$.
\end{proof}

If $f \colon \cX \to S$ is a flat conic bundle over a surface~$S$ (not necessarily proper) with a smooth 
discriminant curve~$\Delta \subset S$, 
consider the preimage~$\cX_\Delta := f^{-1}(\Delta)$, its 
normalization~$\cX_\Delta^\nu \to \cX_\Delta$,
and the Stein factorization
\begin{equation*}
\cX_\Delta^\nu \longrightarrow \tilde\Delta \longrightarrow \Delta.
\end{equation*}
Then the first arrow is a~$\PP^1$-bundle and the second arrow is an \'etale 
double covering (because~$\Delta$ was assumed to be smooth).
We will say that the \'etale covering~$\tilde\Delta \to \Delta$ is \textsf{the 
discriminant double covering} of the conic bundle~$f$.

\begin{lemma}
\label{lemma:tilde-gamma}
The discriminant double covering of the conic bundle~$f \colon \tX^{\ppl} \to 
\PP(A)$ has the form
\begin{equation*}
\tilde\Gamma \cong \Gamma \times_\kk \kk' \longrightarrow \Gamma.
\end{equation*}
\end{lemma}
\begin{proof}
If the conic~$f^{-1}([a])$ is singular, it is a union of two components that 
correspond to the two factors~$\PP(\bar{V}_i)$ in~$Y^+_{\kk'}$
and each of them is contracted by appropriate projection~$Y^+_{\kk'} \to 
\PP(\bar{V}_i)$.
Therefore, the discriminant double covering~$\tilde\Gamma \to \Gamma$ becomes trivial after the extension of scalars to~$\kk'$,
while it is non-trivial over~$\kk$, hence the claim.
\end{proof} 

\subsection{Unirationality}
\label{subsec:unirat-x33}

In this section we prove unirationality of~$X$ assuming that~$X(\kk) \ne 
\varnothing$.
We start with the following observation which might be useful in other 
situations.

\begin{lemma}
\label{lemma:unirationality-22}
Let~$Y$ be a $\kk$-form of~$\PP^2 \times \PP^2$ and let~$W \subset Y$ 
be a $\kk$-rational $\kk$-form of a quadric surface~$\PP^1 \times \PP^1 \subset \PP^2 \times \PP^2$.
Any geometrically irreducible normal divisor $Z \subset Y$ of bidegree~$(2,2)$ 
such that $W \subset Z$ is~$\kk$-unirational.
\end{lemma}

\begin{proof}
Consider the toric birational isomorphism
\begin{equation*}
\xymatrix{
& \Bl_{\PP^1 \times \PP^1}(\PP^2 \times \PP^2) \ar[dl] \ar@{=}[r] &
\Bl_{\PP^1 \sqcup \PP^1}(\PP^4) \ar[dr]
\\
\PP^2 \times \PP^2 \ar@{-->}[rrr]^\chi &&&
\PP^4.
}
\end{equation*}
analogous to the birational transformation of Theorem~\ref{proposition:toric-link} (see also~\cite[Proposition~3]{Zak07}).
Here the map~$\chi$ is the projection from the linear span of~$W$ under the Segre embedding~$\PP^2 \times \PP^2 \subset \PP^8$,
and the right arrow is the blowup of two skew lines in~$\PP^4$.
Denoting by~$\hat{e}$ the class of the exceptional divisor of the left blowup,
by~$\hat{h}$ the hyperplane class of~$\PP^4$, and by~$e_1$ and~$e_2$ the classes 
of the exceptional divisors of the right blowup,
it is easy to check that we have the relations
\begin{equation*}
\begin{cases}
\hbox to .85em{$\hat{h}$} = h_1 + h_2 - \hat{e},\\
\hbox to .85em{$e_1$} = h_1 - \hat{e},\\
\hbox to .85em{$e_2$} = h_2 - \hat{e},
\end{cases}
\qquad \text{and} \qquad
\begin{cases}
\hbox to .95em{$h_1$} = \hat{h} - e_2,\\ 
\hbox to .95em{$h_2$} = \hat{h} - e_1,\\
\hbox to .95em{$\hat{e}$} = \hat{h} - e_1 - e_2.
\end{cases}
\end{equation*}
In particular,
the map~$\chi$ is given by the linear system~$|h_1 + h_2 - \hat{e}|$, hence it is defined over~$\kk$.
Furthermore, we have $2h_1 + 2h_2 - \hat{e} = 3\hat{h} - e_1 - e_2$, and since~$Z$ is normal,
the strict transform of~$Z$ under the map~$\chi$ is a cubic 
threefold~$\hat{Z} \subset \PP^4$ 
passing through the pair of skew lines~$\PP^1 \sqcup \PP^1$.
Moreover, we have $\hat{e} = \hat{h} - e_1 - e_2$, hence the image of~$\hat{E}$ 
is the hyperplane section of this cubic threefold (by the linear span of these lines).
On the other hand, $\hat{E}$ is birational to the $\kk$-rational surface~$W$, hence it is~$\kk$-rational.
In particular, $\hat{Z}(\kk) \ne \varnothing$.

Now, if~$\hat{Z}$ is not a cone, it is $\kk$-unirational by Koll\'ar's theorem~\cite[Theorem~1.2]{Kollar:cubic}.
Otherwise, if~$\hat{Z}$ is a cone and its vertex lies away from the hyperplane spanned by the two skew lines,
then the base of the cone is the $\kk$-rational surface~$\hat{E}$, hence
the cone~$\hat{Z}$ is also~$\kk$-rational.
Finally, if the vertex of the cone lies on~$\hat{E}$, then~$\hat{E}$ itself must be a cubic cone in~$\PP^3$, 
and since it also contains two skew lines, it is not geometrically irreducible, which is absurd.
\end{proof}

Now we can deduce unirationality of~$X$.

\begin{proposition}
\label{prop:x33}
If $X$ is a Fano threefold of type~\type{3,3} with~$X(\kk) \ne \varnothing$
then~$X$ is $\kk$-unirational.
\end{proposition}

\begin{proof}
Let $x_0$ be a $\kk$-point on~$X$.

First, assume~$\rF_1(X,x_0) = \varnothing$.
By Proposition~\ref{prop:conic-bundle-x33} we have a $\kk$-birational 
map~\mbox{$X\dashrightarrow \tX^{\ppl}$}, 
where~$f \colon \tX^{\ppl} \to \PP(A)$ is a conic bundle.
Moreover, the $\kk$-rational surface~\mbox{$E \cong \PP(T_{x_0}X) \subset \tX$} dominates the base of this conic bundle.
Therefore, $\tX$ is $\kk$-unirational
(see, e.g. \cite[Lemma~4.14(i)]{KP19}); and hence so is~$X$.

Now assume that $\rF_1(X,x_0) \ne \varnothing$.
By Proposition~\ref{prop:x33-plus} we have a birational 
map~$X\dashrightarrow X^+$ 
where~$X^+$ is a geometrically irreducible normal divisor of bidegree~$(2,2)$ in a $\kk$-form of~$\PP^2 \times \PP^2$
that contains a $\kk$-form of a $\kk$-rational quadric surface~$\PP^1 \times \PP^1$.
Therefore, $X^+$ is $\kk$-unirational by Lemma~\ref{lemma:unirationality-22}, 
hence so is~$X$.
\end{proof}

\subsection{Non-rationality}

In this section we prove non-rationality of Fano threefolds of type~\type{3,3}.
We will use the following reformulation of a result of Benoist--Wittenberg 
from~\cite{BW}.

\begin{theorem}
\label{prop:non-rationality-conic-bundle}
Let $\cX \to S$ be a flat conic bundle over a smooth $\kk$-rational surface~$S$ 
with smooth connected discriminant curve~$\Delta \subset S$.
Assume the discriminant double covering takes the form
\begin{equation*}
\tilde\Delta \cong \Delta \times_\kk \kk' \longrightarrow \Delta
\end{equation*}
where $\kk'/\kk$ is a quadratic extension of the base field.
If the conic bundle~$\cX_{\kk'} \to S_{\kk'}$ admits a rational section 
and the curve~$\Delta$ is not hyperelliptic then~$\cX$ is not~$\kk$-rational.
\end{theorem}

Note that we require neither the surface~$S$ nor the curve~$\Delta$ to be proper; 
moreover, during the proof we will further shrink~$S$ but keep (the generic point of) the curve~$\Delta$ in~$S$.

\begin{proof}
Since $S$ is normal and $f$ is proper any rational section of~$f$ extends to codimension~$1$ points, 
hence defines a regular section over the complement of a finite subscheme of~$S$.
Moreover, over the complement of this finite subscheme 
the section does not pass through singular points of fibers of~$f$,
hence it defines a section of the morphism $\cX_\Delta^\nu \to \Delta$, 
where recall that~$\cX_\Delta^\nu$ is the normalization of~$\cX_\Delta = 
f^{-1}(\Delta)$. 
Therefore it also gives a section of the discriminant double covering~$\tilde\Delta \to \Delta$.
If the original section is defined over~$\kk$, we obtain a contradiction with the isomorphism~$\tilde\Delta \cong \Delta \times_\kk \kk'$;
this means that the morphism~$f$ has no rational sections defined over~$\kk$.

Now consider a rational section of $f \colon \cX_{\kk'} \to S_{\kk'}$.
Removing if necessary a finite subscheme from~$S$ we may assume that this section is regular.
Its intersection with the conjugate section (with respect to the~$\Gal(\kk'/\kk)$-action) 
projects to a curve in~$S$ which is disjoint from~$\Delta$
(because a regular section does not pass through singular points of fibers).
So, shrinking~$S$ further we may assume that the section and its conjugate do not intersect.
Then the union
\begin{equation*}
Z \subset \cX
\end{equation*}
of the section and its conjugate is a $2$-section of~$f$ defined over~$\kk$;
moreover, $Z \cong S \times_\kk \kk'$ and in particular~$Z$ is \'etale over~$S$. 

Consider the bundles $\cV := (f_*\omega_\cX^{-1})^\vee$ of rank~$3$
and~$\cV_Z := (f_*\omega_\cX^{-1}\vert_Z)^\vee$ of rank~$2$ on~$S$.
The restriction morphism~$\omega_\cX^{-1} \to \omega_\cX^{-1}\vert_Z$ induces an embedding of vector bundles~$\cV_Z \hookrightarrow \cV$ 
and a Cartesian square
\begin{equation*}
\xymatrix{
Z \ar[r] \ar[d] &
\PP_S(\cV_Z) \ar[d] 
\\
\cX \ar[r] &
\PP_S(\cV),
}
\end{equation*}
where all arrows are the natural embeddings.

Shrinking the surface~$S$ again but keeping an open part of the curve~$\Delta$ 
in it
we may assume that the bundles~$\cV$ and~$\cV_Z$ are trivial and that
the subvarieties~\mbox{$\cX \subset \PP_S(\cV)$} and~$Z \subset \PP_{S}(\cV_Z)$
are given by a quadratic form~$q \in \Sym^2\cV^\vee$ and its restriction $q_Z 
\in \Sym^2\cV_Z^\vee$ to~$\cV_Z$, respectively.
Since $Z$ is \'etale over~$S$, the form~$q_Z$ is everywhere non-degenerate and 
can be written as follows
\begin{equation*}
q_Z = x^2 - \alpha y^2,
\end{equation*}
where $(x,y)$ are homogeneous coordinates in the fiber of~$\PP_S(\cV_Z) 
\cong S \times \PP^1$ 
and $\alpha \in \kk^\times$ is such that~$\kk' = \kk(\sqrt{\alpha})$.
Now, considering the orthogonal complement to~$\cV_Z$ in~$\cV$ we see that~$q$ 
takes the form 
\begin{equation*}
q = x^2 - \alpha y^2 - Fz^2,
\end{equation*}
where $F$ is an equation of~$\Delta$ on~$S$.
Thus, the conic bundle $\cX \to S$ is birational to conic bundles considered in~\cite[\S3.3.1]{BW}, 
hence~$\cX$ is not~$\kk$-rational by~\cite[Proposition~3.4]{BW}.
\end{proof}

Now we apply this to prove non-rationality of threefolds of type~\type{3,3}.

\begin{corollary}
\label{cor:x33}
If $X$ is a Fano threefold of type~\type{3,3} then~$X$ is not~$\kk$-rational.
\end{corollary}
\begin{proof}
If $X$ is not~$\kk$-unirational, there is nothing to prove.
So, assume~$X$ is $\kk$-unirational. 
Then there exists a $\kk$-point $x_0 \in X$ such that~$\rF_1(X,x_0) = \varnothing$.
Consider the conic bundle $\tX^{\ppl} \to \PP^2$ constructed in Proposition~\ref{prop:conic-bundle-x33}.
By Lemma~\ref{lemma:gamma-x33} the discriminant curve of~$f$ is the smooth non-hyperelliptic curve~$\Gamma$ defined in~\S\ref{subsec:discriminant}
and by Lemma~\ref{lemma:tilde-gamma} the discriminant double covering has the form~$\tilde\Gamma \cong \Gamma\times_\kk \kk'$.
Finally, the description of Proposition~\ref{prop:conic-bundle-x33} shows that~$f$ admits a rational section after base change to~$\kk'$.
Therefore, Theorem~\ref{prop:non-rationality-conic-bundle} applies and proves that~$\tX^{\ppl}$ is not~$\kk$-rational,
hence~$X$ is not $\kk$-rational as well.
\end{proof}


\section{Fano threefolds of type \type{1,1,1,1}}
\label{sec:x1111}

In this section we apply the degeneration technique of~\cite{NS} to prove Theorem~\ref{thm:x1111-non-st-rat}.

\subsection{Toric degeneration}

To start with we consider $\rY_0 = (\PP^1)^4$, denote by $(u_i:v_i)$ the homogeneous coordinates on the $i$-th factor,
and consider the point
\begin{equation*}
\ry_0 := (1,1,1,1) \in \rY_0.
\end{equation*}
Clearly, $\rY_0$ is a toric variety with respect to the action of the split torus $\Gm^4$ that rescales the~$v_i$.
We also consider the action of~$\fS_4$ on~$\rY_0$ that permutes the factors.
It normalizes the torus action, and together they generate an action of 
the group~$\Gm^4 \rtimes \fS_4$.
Finally, consider the subtorus 
\begin{equation}
\label{eq:t0}
\rT_0 := \big\{ (t_1,t_2,t_3,t_4) \in \Gm^4 \mid t_1t_2t_3t_4 = 1 \big\}
\end{equation}
and the collection of three 1-parametric subgroups
\begin{equation}
\label{eq:rt-4}
\rT_0^{i_1,i_2;i_3,i_4} := \big\{ (t_1,t_2,t_2,t_4) \mid t_{i_1} = t_{i_2} = t_{i_3}^{-1} = t_{i_4}^{-1} \big\} \subset \rT_0,
\end{equation}
where~$(i_1,i_2)(i_3,i_4) \in \rV_4 \setminus \{1\} \subset \fS_4$ is a nontrivial element of the Klein subgroup.

\begin{lemma}
\label{lemma:x1111-toric}
The subvariety
\begin{equation*}
\rX_0^{\mathrm{toric}} := \{ u_1u_2u_3u_4 - v_1v_2v_3v_4 = 0 \} \subset \rY_0
\end{equation*}
is the unique $\rT_0$-invariant divisor of multidegree~$(1,1,1,1)$ in~$\rY_0$ which contains the point~$\ry_0$.
It is a toric variety with~$6$ ordinary double points
\begin{equation*}
\rx_{p,q} = \big\{ ((u_1:v_1),(u_2:v_2),(u_3:v_3),(u_4:v_4)) \mid 
\text{$u_i = 0$ if~$i \in \{p,q\}$, and $v_i = 0$ if~$i \not\in \{p,q\}$} \big\}.
\end{equation*}
For each permutation~$(i_1,i_2)(i_3,i_4) \in \rV_4 \setminus \{1\} \subset \fS_4$ the curve
\begin{equation}
\label{eq:c-iiii}
\rC_{i_1,i_2;i_3,i_4} := \overline{\rT_0^{i_1,i_2;i_3,i_4}\cdot \ry_0} \subset \rX_0^{\mathrm{toric}}
\end{equation}
is a smooth rational curve and~$\rC_{i_1,i_2;i_3,i_4} \cap \Sing(\rX_0^{\mathrm{toric}}) = \{ \rx_{i_1,i_2}, \rx_{i_3,i_4} \}$.
\end{lemma}

\begin{proof}
The monomial basis in the space of homogeneous polynomials of multidegree~$(1,1,1,1)$
is a weight basis for the action of~$\rT_0$, 
and all weights are different except for the weight~$0$ which has multiplicity~2 
and the corresponding weight space is spanned by the monomials~$u_1u_2u_3u_4$ and~$v_1v_2v_3v_4$.
Therefore, every $\rT_0$-invariant divisor is either given by a monomial equation 
(but then it does not contain the point~$\ry_0$),
or it is given by a linear combination of~$u_1u_2u_3u_4$ and~$v_1v_2v_3v_4$,
and if it contains the point~$\ry_0$, it is equal to~$\rX_0^{\mathrm{toric}}$.
The latter is obviously a toric variety with respect to the natural action of~$\rT_0$.

In the affine chart $v_1 \ne 0$, $v_2 \ne 0$, $u_3 \ne 0$, $u_4 \ne 0$ we can set $v_1 = v_2 = u_3 = u_4 = 1$ 
and use~$u_1$, $u_2$, $v_3$, $v_4$ as coordinates.
Then the equation of~$\rX_0^{\mathrm{toric}}$ takes the form 
\begin{equation*}
u_1u_2 - v_3v_4 = 0,
\end{equation*}
which means that the origin of the chart, i.e., the point $\rx_{3,4} = (0,0,\infty,\infty) \in \rY_0$ 
is an ordinary double point of~$\rX_0^{\mathrm{toric}}$.
Considering similarly the other charts, we see that the singular locus of~$\rX_0^{\mathrm{toric}}$ 
is the $\fS_4$-orbit of the point~$\rx_{3,4}$;
in particular each singular point of~$\rX_0^{\mathrm{toric}}$ is an ordinary double point.
Moreover,  we see that the hypersurface~$\rX_0^{\mathrm{toric}}$ is normal.

The orbits of the point~$\ry_0$ under the 1-parametric subgroups~\eqref{eq:rt-4} are
\begin{equation*}
\{ (t,t,t^{-1},t^{-1}) \mid t \in \Gm \},
\qquad 
\{ (t,t^{-1},t,t^{-1}) \mid t \in \Gm \},
\qquad 
\{ (t,t^{-1},t^{-1},t) \mid t \in \Gm \}.
\end{equation*}
It is easy to see that the closure of the first orbit is smooth and contains the point~$\rx_{1,2}$ and~$\rx_{3,4}$,
and similarly for the other two orbits.
\end{proof}

For a field extension~$\kk'/\kk$ we denote by~$\Res_{\kk'/\kk} \colon \Sch_{\kk'} \to \Sch_\kk$ 
the Weil restriction of scalars functor from the category of~$\kk'$-schemes to the category of $\kk$-schemes, 
the right adjoint to the extension of scalars~\mbox{$- \otimes_\kk \kk' \colon \Sch_\kk \to \Sch_{\kk'}$}.
Consider the projective line~$\PP^1_{\kk'}$, 
the torus~$\Gm$ acting faithfully on~$\PP^1_{\kk'}$ and denote by~$0,\infty \in \PP^1_{\kk'}$ its fixed points.

\begin{proposition}
\label{prop:forms-x44}
For a field extension~$\kk'/\kk$ of degree~$4$ consider the $\kk$-forms
\begin{equation}
\label{eq:y-t-res}
Y := \Res_{\kk'/\kk}(\PP^1_{\kk'}),
\qquad 
T := \Ker\left(\Res_{\kk'/\kk}\Gm \longrightarrow \Gm\right)
\end{equation}
of~$\rY_0 = (\PP^1)^4$ and of the torus~$\rT_0$, and the natural faithful $T$-action on~$Y$.
Let
\begin{equation*}
y \in Y
\end{equation*}
be the $\kk$-point that corresponds to the point~$1 \in \PP^1_{\kk'}$.
Then 
\begin{enumerate}
\item 
\label{item:x-varpi}
The half-anticanonical linear system of~$Y$ is defined over~$\kk$
and it contains a unique $T$-invariant divisor~$X^{\mathrm{toric}} \subset Y$ passing through~$y$.
\item 
\label{item:x-sing}
The divisor~$X^{\mathrm{toric}}$ is integral and has ordinary double points in the sense of~\cite[Definition~4.2.1]{NS} 
with the singular locus of length~$6$.
\end{enumerate}
\end{proposition}
 
\begin{proof}
The fact that~$Y$ and~$T$ are $\kk$-forms of~$\rY_0$ and~$\rT_0$ is obvious from the definition of Weil restriction of scalars,
and the $\kk$-point~$y$ is obtained from the extension--restriction adjunction.
Note that upon extension of scalars to~$\bkk$, 
the triple~$(Y,T,y)$ becomes isomorphic to the triple~$(\rY_0,\rT_0,\ry_0)$.

\ref{item:x-varpi}
Let~$H$ denote the Segre class of~$Y$ (the half of the anticanonical class);
it is obviously Galois-invariant, and since~$Y$ has a $\kk$-point, $H$ is defined over~$\kk$.
Therefore, the linear system
\begin{equation}
\label{eq:fp}
\fP = |H - y| \cong \PP^{14}
\end{equation}
of divisors in~$|H|$ containing~$y$ is defined over~$\kk$.
We define~$X^{\mathrm{toric}} \subset Y$ as the closure of the~$T$-orbit of the point~$y$;
the uniqueness of~$X^{\mathrm{toric}}$ follows from Lemma~\ref{lemma:x1111-toric}.

\ref{item:x-sing}
The extension of scalars of~$X^{\mathrm{toric}} \subset Y$ to~$\bkk$ coincides with~\mbox{$\rX_0^{\mathrm{toric}} \subset \rY_0$}, 
hence its singular locus~$Z := \Sing(X^{\mathrm{toric}})$ has length~$6$ 
and if~$E$ is the exceptional divisor of the blowup
\begin{equation*}
\tX := \Bl_{Z}(X^{\mathrm{toric}})
\end{equation*}
then~$E \to Z$ is a smooth quadric bundle.
According to~\cite[Definition~4.2.1]{NS} it remains to check that this bundle has a section.
For this note that the union of the 1-parametric subgroups~$\rT_0^{i_1,i_2;i_3,i_4} \subset \rT_0$ defined in~\eqref{eq:rt-4}
is Galois-invariant, hence it comes from a $\kk$-subset in the torus~$T$,
and hence the closure of the image of the point~$y$ under the action of this subset
is a curve~$C \subset X$ defined over~$\kk$.
Furthermore, the extension of scalars of~$C$ to~$\bkk$ is the union of the curves~\eqref{eq:c-iiii}.
In particular, the curve~$C$ contains the singular locus~$Z$,
and the intersection of its strict transform to the blowup~$\tX$ with the exceptional divisor~$E$ 
provides a section for~\mbox{$E \to Z$}. 
\end{proof}

Now let~$X$ be a smooth Fano threefold of type~\type{1,1,1,1}.
Recall the definition~\eqref{eq:gx} of the Galois group~$\rG_X \subset \fS_4$ of~$X$.
In the next lemma we use notation introduced in Proposition~\ref{prop:forms-x44}.

\begin{lemma}
\label{lem:x1111-models}
If~$\kk'/\kk$ is the field extension of degree~$4$  associated with an epimorphisma 
$\rG(\bkk/\kk) \twoheadrightarrow\rG$ onto a transitive subgroup $\rG\subset \fS_4$,
then a general divisor~$X \subset Y$ from the linear system~\eqref{eq:fp} 
is a smooth Fano threefold of type~\type{1,1,1,1} with~$\rG_X = \rG$, $\uprho(X) = 1$, and~$X(\kk) \ne \varnothing$.
\end{lemma}

\begin{proof}
The smoothness of a general divisor~$X$ in the linear system~$\fP$ follows from the Bertini theorem, 
the property~$X(\kk) \ne \varnothing$ is obvious because~$X$ contains the $\kk$-point~$y$,
the equality~$\rG_X = \rG$ follows from the construction, and~$\uprho(X) = 1$ follows from transitivity of~$\rG \subset \fS_4$.
\end{proof}

\begin{remark}
One can also prove the converse statement: any Fano threefold~$X$ of type~\type{1,1,1,1} 
with~$X(\kk) \ne \varnothing$ and~$\rG_X = \rG$ is isomorphic to a divisor in the linear system~$\fP$,
see~\cite[Proposition~7.16]{K22}.
\end{remark}

Now we are ready to prove Theorem~\ref{thm:x1111-non-st-rat}.

\begin{proof}[Proof of Theorem~\xref{thm:x1111-non-st-rat}]
We consider the field extension~$\kk'/\kk$ as in Lemma~\ref{lem:x1111-models}
and use the construction and notation of Proposition~\ref{prop:forms-x44};
in particular the linear system~$\fP \cong \PP^{14}_\kk$ of half-anticanonical divisors in~$Y$.
Let~$\mathfrak{p}_0 \in \fP$ be the point that corresponds to the toric divisor~$X^{\mathrm{toric}} \subset Y$.
Note that the space of lines in~$\fP$ through the point~$\mathfrak{p}_0$ is the projective space~$\PP^{13}_\kk$,
in particular $\kk$-points are Zariski dense in it.
Therefore, there is a line~$L \subset \fP$ through~$\mathfrak{p}_0$ defined over~$\kk$.
We denote by
\begin{equation*}
\cX \to L
\end{equation*}
the corresponding family of half-anticanonical divisors in~$Y$.
Then the general point of~$L$ corresponds to a smooth variety~$\cX_L$ of type~\type{1,1,1,1} 
over the field~$\kk(L) \cong \kk(t)$.

Since for a general $\kk$-point $\mathfrak{p} \in L$ 
the fiber of~$\cX_{\mathfrak{p}}$ has $\rG_{X_{\mathfrak{p}}} = \rG$ by Lemma~\ref{lem:x1111-models}, 
and since the natural restriction morphism of Galois groups~$\Gal(\overline{\kk(L)}/\kk(L)) \to \Gal(\bkk/\kk)$ is surjective,
we have~$\rG_{\cX_L} = \rG$.
Since~$\rG \subset \fS_4$ is transitive, this implies~$\uprho(\cX_L) = 1$.
Finally, $\cX_L$ by construction contains the $\kk(L)$-point~$y \times_\kk {\kk(L)}$, hence~$\cX_L(\kk(L)) \ne \varnothing$.

Assume~$\cX_L$ is stably rational.
Consider the point~$\mathfrak{p}_0 \in L$ as a special point of the family~$\cX/L$.
By Proposition~\ref{prop:forms-x44} the corresponding variety~$\cX_{\mathfrak{p}_0} = X^{\mathrm{toric}}$ 
is integral with ordinary double points,
hence by~\cite[Proposition~4.2.9]{NS} the family~$\cX/L$ is~$\mathbb{L}$-faithful in the sense of~\cite[Definition~4.2.7]{NS}.
Therefore, by~\cite[Proposition~4.2.10]{NS} the special fiber~$X^{\mathrm{toric}}$ is stably rational.

On the other hand, by definition the Galois group of the extension~$\kk'/\kk$ coincides with~$\rG_X$
and contains the Klein group~$\rV_4$.
By~\cite[\S~2.4.8]{Vos} for any smooth compactification~$V\supset T$
one has 
\begin{equation*}
H^1(\Gal(\bkk/\kk),\, \Pic(V_{\bkk})) = H^1(\rG_X,\, \Pic(V_{\bkk}))\neq 0. 
\end{equation*}
Since this
group is a stable birational invariant 
(see, e.g., \cite[\S4.4]{Vos}, or~\cite[\S2.A]{Colliot-Thelene-Sansuc-1987}),
the torus~$T$ and the corresponding toric variety~$X^{\mathrm{toric}}$ are not stably rational.
This contradiction shows that~$\cX_L$ is not stably rational and completes the proof of the theorem.
\end{proof}

\appendix

\section{Constructing morphisms of Hilbert schemes}

In this section we show how one can use technique of derived categories to construct morphisms of Hilbert schemes.
For smooth projective varieties~$X$ and~$Y$ we denote by~$\pi_X$ and~$\pi_Y$ the projections from~$X \times Y$ to the factors, 
and for an object~$\cK \in \Db(X \times Y)$ we denote by
\begin{equation*}
\Phi_\cK \colon \Db(X) \longrightarrow \Db(Y),
\qquad 
\cF \longmapsto \bR\pi_{Y*}(\bL\pi_X^*(\cF) \otimes^\bL \cK)
\end{equation*}
the corresponding Fourier--Mukai functor from the bounded derived category~$\Db(X)$ of coherent sheaves on~$X$ to that of~$Y$.
For an integral valued polynomial~$p \in \QQ[t]$ we denote by~$\Hilb_{p}(X)$ 
the Hilbert scheme of subschemes in~$X$ with Hilbert polynomials~$p$ with respect to 
a given polarization.

\begin{proposition}
\label{prop:hilb-map}
Let~$X$ and~$Y$ be smooth projective varieties.
If~$\cK \in \Db(X \times Y)$ is an object such that for any subscheme~$Z \subset X$ with Hilbert polynomial~$p$ 
the object~$\Phi_\cK(\cO_Z) \in \Db(Y)$ is isomorphic to the structure sheaf of a point~$y(Z) \in Y$ 
then there is a morphism of schemes
\begin{equation*}
\varphi \colon \Hilb_{p}(X) \longrightarrow Y
\end{equation*}
such that~$\varphi([Z]) = y(Z)$.
\end{proposition}

\begin{proof}
Let~$Z \subset X \times S$ be a family of subschemes in~$X$ flat over $S$ with Hilbert polynomial~$p$.
Let 
\begin{equation*}
\cF := \Phi_{\pi_{XY}^*\cK}(\cO_Z) 
= \bR\pi_{SY*}(\bL\pi_{XS}^*(\cO_Z) \otimes^\bL \pi_{XY}^*\cK)
\in \Db(S \times Y)
\end{equation*}
be the image of the structure sheaf of~$Z$ under the induced Fourier--Mukai functor from~$\Db(X \times S)$ to~$\Db(S \times Y)$,
where~$\pi_{XS}$, $\pi_{SY}$ and~$\pi_{XY}$ are the projections of~$X \times S \times Y$ to the pairwise products of factors. 
By base change and the projection formula, for each point~$s \in S$ we have 
\begin{equation*}
i_s^*\cF \cong \Phi_\cK(\cO_{Z_s}),
\end{equation*}
where $i_s \colon \{s\} \times Y \hookrightarrow S \times Y$ is the natural embedding
and~$Z_s \subset X$ is the fiber of~$Z$ over~$s \in S$.
Thus, we have~$i_s^*\cF \cong \cO_{y(Z_s)}$ by assumption, therefore by~\cite[Lemma~4.4(iii)]{K19}
there is a unique morphism~$\varphi_S \colon S \to Y$ such that~$\cF$ 
is isomorphic up to twist to the structure sheaf of the graph of~$\varphi_S$;
in particular, $\varphi_S(s) = y(Z_s)$ for each~$s \in S$.
Now applying this argument to~$S = \Hilb_p(X)$ and~$Z$ the universal subscheme, we obtain the required morphism~$\varphi$.
\end{proof}

In~\S\ref{subsec:conics} we apply Proposition~\ref{prop:hilb-map} 
to the Hilbert scheme of conics on the threefold~$X \cong \Bl_{\Gamma_1}(Q_1)$,
where~\mbox{$Q_1 \subset \PP^4$} is a smooth quadric and~$\Gamma_1 \subset Q_1$ is a linearly normal smooth rational quartic curve.
Recall that~$\rF_1(X)$ and~$\rF_2(X)$ denote the Hilbert schemes of lines and conics on~$X$, 
and that there is a natural embedding~$\Gamma_1 \subset \rF_1(X)$ 
of a connected component (see Lemma~\ref{lemma:lines-1}).
Recall also that the second component~$\Gamma_2 \subset \rF_1(X)$ corresponds to lines on~$Q_1$ bisecant to~$\Gamma_1$.

\begin{corollary}
\label{cor:f2x-gamma1}
There is a morphism~$\varphi_1 \colon \rF_2(X) \to \Gamma_1$ such that for a smooth conic~$C \subset X$ 
one has~$\varphi_1([C]) = [L]$, where~$L \subset X$ is the unique line corresponding to a point of~$\Gamma_1$ such that~$C \cap L \ne \varnothing$.
Moreover, if~$C = L_1 \cup L_2$ is a reducible conic, so that~$L_1 \cap L_2 \ne \varnothing$ and if~$L'_2$ 
is the other line corresponding to a point of~$\Gamma_1$ such that~$L_1 \cap L'_2 \ne \varnothing$, then~$\varphi_1([C]) = [L'_2]$.
\end{corollary}

\begin{proof}
Let~$\pi_1 \colon X \to Q_1$ be the blowup morphism and let~$E_1 \subset X$ be its exceptional divisor;
note that~$E_1$ is the universal family of lines on~$X$ over the connected component~$\Gamma_1 \subset \rF_1(X)$ of the Hilbert scheme of lines.
Let~$\varepsilon \colon E_1 \to X \times \Gamma_1$ be the corresponding embedding and consider 
\begin{equation*}
\cK := \varepsilon_*\cO_{E_1}(E_1) \in \Db(X \times \Gamma_1).
\end{equation*}
Let us check that the assumption of Proposition~\ref{prop:hilb-map} is satisfied for~$Y = \Gamma_1$.

If~$C \subset X$ is a smooth conic, then~$C \cdot E_1 = 1$, and~$C \not\subset E_1$, 
therefore~$C \cap E_1 = \{x\}$ is a single point and the intersection is transverse.
Therefore~$\cO_C \otimes^\bL \cO_{E_1}(E_1) \cong \cO_x$, hence~$\Phi_\cK(\cO_C) \cong \cO_{\pi_1(x)}$.

If~$C = L_1 \cup L_2$ is a reducible conic, so that~$L_2 \subset E_1$ and~$L_1 \cap E_1 = \{x,x'\}$ with~$L_1 \cap L_2 = \{x\}$
(if~$L_1$ is tangent to~$E_1$ we take~$x' = x$), 
then using the exact sequences 
\begin{equation*}
0 \longrightarrow \cO_{L_1}(-1) \longrightarrow \cO_C \longrightarrow \cO_{L_2} \longrightarrow 0
\qquad\text{and}\qquad
0 \longrightarrow \cO_{L_2}(-1) \longrightarrow \cO_C \longrightarrow \cO_{L_1} \longrightarrow 0
\end{equation*}
it is easy to check that~$\cO_C \otimes^\bL \cO_{E_1}(E_1)$ fits into a distinguished triangle
\begin{equation*}
\cO_C \otimes^\bL \cO_{E_1}(E_1) \longrightarrow \cO_{L_2}(-1) \oplus \cO_{x'} \longrightarrow \cO_{L_2}(-1)[2]
\end{equation*}
(if~$L_1$ is tangent to~$E_1$ the middle term should be replaced by an extension of~$\cO_{L_2}(-1)$ by~$\cO_{x'}$).
Since the pushforward functor~$\bR\pi_{1*}$ kills the sheaf~$\cO_{L_2}(-1)$, it follows that~$\Phi_\cK(\cO_C) \cong \cO_{\pi_1(x')}$.
 
Now, applying Proposition~\ref{prop:hilb-map} we conclude that there is a morphism~$\varphi_1 \colon \rF_2(X) \to \Gamma_1$ 
such that~$\varphi_1([C]) = \pi_1(x)$ if~$C$ is smooth 
and~$\varphi_1([L_1 \cup L_2]) = \pi_1(x')$, 
with the notation for points~$x$ and~$x'$ introduced above.
\end{proof} 

\def\cprime{$'$}


\end{document}